\documentclass[preprint]{elsarticle}

\journal{Journal of Computational Physics}

\def\d{\delta} 
 
\def\e{{\epsilon}}

\def\norm#1{\|#1\|} 
\def\wb{{\bar w}}

\usepackage{algorithm}
\usepackage{algorithmic}
\usepackage{amssymb}
\usepackage{amsmath}
\usepackage{amsthm}
\usepackage{caption}
\usepackage{comment}
\usepackage{float}
\usepackage[latin1]{inputenc}
\usepackage[pdfborder={0 0 0}, colorlinks=true, linkcolor=blue, citecolor=red]{hyperref}\hypersetup{pdfborder=0 0 0}
\usepackage{mathrsfs}
\usepackage{multicol}
\usepackage{multirow}
\usepackage{subfigure}
\usepackage{wrapfig}
\usepackage{ulem}

\usepackage{soul,xargs}
\usepackage[pdftex,dvipsnames]{xcolor}

\usepackage[colorinlistoftodos,prependcaption,textsize=tiny]{todonotes}
\newcommandx{\unsure}[2][1=]{\todo[linecolor=red,backgroundcolor=red!25,bordercolor=red,#1]{#2}}
\newcommandx{\change}[2][1=]{\todo[linecolor=blue,backgroundcolor=blue!25,bordercolor=blue,#1]{#2}}
\newcommandx{\info}[2][1=]{\todo[linecolor=OliveGreen,backgroundcolor=OliveGreen!25,bordercolor=OliveGreen,#1]{#2}}
\newcommandx{\improvement}[2][1=]{\todo[linecolor=Plum,backgroundcolor=Plum!25,bordercolor=Plum,#1]{#2}}
\newcommandx{\thiswillnotshow}[2][1=]{\todo[disable,#1]{#2}}

\newtheorem{theorem}{Theorem}
\newtheorem{corollary}{Corollary}
\newtheorem{lemma}{Lemma}
\newtheorem{remark}{Remark}
\graphicspath{{./plots/}}

\begin{document}

\newcommand{\TODO}[1]{ \fbox{\parbox{3in}{\bf TODO: #1}}}

\newcommand{\grbf}[1] {\mbox{\boldmath${#1}$\unboldmath}}
\newcommand{\gbf}[1] {\mathbf{#1}}

\newcommand{\beq} {\begin{equation}}
\newcommand{\eeq} {\end{equation}}
\newcommand{\bdm} {\begin{displaymath}}
\newcommand{\edm} {\end{displaymath}}
\newcommand{\bit}{\begin{itemize}}
\newcommand{\eit}{\end{itemize}}
\newcommand{\bde}{\begin{description}}
\newcommand{\ede}{\end{description}}
\newcommand{\bce}{\begin{center}}
\newcommand{\ece}{\end{center}}
\newcommand{\ben} {\begin{enumerate}}
\newcommand{\een} {\end{enumerate}}
\newcommand{\bea} {\begin{eqnarray}}
\newcommand{\eea} {\end{eqnarray}}
\newcommand{\barr} {\begin{array}}
\newcommand{\earr} {\end{array}}
\newcommand{\bean} {\begin{eqnarray*}}
\newcommand{\eean} {\end{eqnarray*}}
\newcommand{\edoc} {\end{document}}
\newcommand{\On}{\hspace{0.2cm} \mbox{on} \hspace{0.2cm}}
\newcommand{\In}{\hspace{0.2cm} \mbox{in} \hspace{0.2cm}}
\newcommand{\Minimize}{\hspace{0.2cm} \mbox{minimize} \hspace{0.2cm}}
\newcommand{\Maximize}{\hspace{0.2cm} \mbox{maximize} \hspace{0.2cm}}
\newcommand{\SubjectTo}{\hspace{0.2cm} \mbox{subject} \hspace{0.15cm}
            \mbox{to} \hspace{0.2cm}}

\def\etal{{\it et al.~}}

\newcommand{\point}[1] {\ensuremath{\boldsymbol{#1}}}
\newcommand{\vvec}[1] {\ensuremath{\boldsymbol{#1}}}
\renewcommand{\vec}[1] {\ensuremath{\boldsymbol{#1}}}
\newcommand{\cvec}[1] {\ensuremath{\boldsymbol{{#1}}}}
\newcommand{\ten}[1] {\ensuremath{\boldsymbol{#1}}}
\newcommand{\tenfour}[1] {\ensuremath{\boldsymbol{\mathsf{#1}}}}
\newcommand{\comp}[1]{\begin{pmatrix}{ #1 }\end{pmatrix}}
\newcommand{\vv}{\ensuremath{\vec{v}}}
\newcommand{\vr}{\left(\ensuremath{\vec{r}}\right)}
\newcommand{\att}{\left(t\right)}
\newcommand{\tvv}{\ensuremath{\tilde{\vec{v}}}}
\newcommand{\ww}{\ensuremath{\vec{w}}}
\newcommand{\GG}{\ensuremath{\ten{G}}}
\newcommand{\FF}{\ensuremath{\boldsymbol{\mathfrak{F}}}}
\newcommand{\tEE}{\ensuremath{\tilde{\vec{E}}}}
\newcommand{\nn}{\ensuremath{\vec{n}}}
\renewcommand{\SS}{\ensuremath{\ten{S}}}
\newcommand{\tSS}{\ensuremath{\tilde{\ten{S}}}}
\newcommand{\cp}{\ensuremath{{c_p}}}
\newcommand{\cs}{\ensuremath{{c_s}}}
\newcommand{\tcs}{\ensuremath{{\tilde{c}_s}}}
\newcommand{\Vp}{\ensuremath{{V^e_p}}}
\newcommand{\Vs}{\ensuremath{{V^e_s}}}
\newcommand{\Ss}{\ensuremath{{S^e_s}}}
\newcommand{\nxnx}[1] {\ensuremath{\nn\times\left(\nn\times{#1}\right)}}

\newcommand{\dd}[2] {\ensuremath{\frac{\partial {#1}}{\partial {#2}}}}
\newcommand{\DD}[2] {\ensuremath{\frac{d {#1}}{d {#2}}}}
\newcommand{\Grad} {\ensuremath{\nabla}}
\newcommand{\Gradv}[1]{\ensuremath{\nabla_{#1}}}
\newcommand{\Div} {\ensuremath{\nabla\cdot}}
\newcommand{\Divv}[1] {\ensuremath{\nabla_{#1}\cdot}}
\newcommand{\Curl} {\ensuremath{\nabla\times}}
\newcommand{\Cten}{\ensuremath{\tenfour{C}}}

\newcommand{\eval}[2][\right]{\relax
  \ifx#1\right\relax \left.\fi#2#1\rvert}

\providecommand{\abs}[1]{\left\lvert#1\right\rvert}
\providecommand{\norm}[1]{\left\lVert#1\right\rVert}

\newcommand{\Nel} {\ensuremath{{N_\text{el}}}}
\newcommand{\Ned} {\ensuremath{{N_\text{ed}}}}
\newcommand{\De} {\ensuremath{{\mathsf{D}^e}}}
\newcommand{\Dep} {\ensuremath{{\mathsf{D}^{e'}}}}
\newcommand{\Dhat} {\ensuremath{\hat{\mathsf{D}}}}
\newcommand{\Dehat} {\ensuremath{{\hat{\mathsf{D}}^e}}}
\newcommand{\half} {\ensuremath{\frac{1}{2}}}
\newcommand{\IN} {\ensuremath{{\mathcal{I}_N}}}
\newcommand{\INe} {\ensuremath{{\mathcal{I}_{N_e}}}}

\newcommand{\mc}[1]{\mathcal{#1}}
\newcommand{\mcC}[1]{\mathcal{#1}}
\newcommand{\mb}[1]{\mathbf{#1}}
\newcommand{\mbb}[1]{\mathbb{#1}}
\newcommand{\mi}[1]{\mathit{#1}}
\newcommand{\mf}[1]{\mathfrak{#1}}
\newcommand{\Oi}[1]{\int_{\mc{D}} #1 \, d\vec{x}}
\newcommand{\TOi}[1]{\int_{0}^T \int_{\mc{D}} #1 \, d\vec{x}dt}
\newcommand{\Ti}[1]{\int_{0}^T #1 \, dt}
\newcommand{\TDei}[1]{\int_{0}^T\int_{\De} #1 \, d\vec{x}dt}
\newcommand{\TDeid}[1]{\int_{0}^T\int_{\De,N} #1 \, d\vec{x}dt}
\newcommand{\Dei}[1]{\int_{\De} #1 \, d\vec{x}}
\newcommand{\DeI}[1]{\int_{D} #1 \, d\vec{x}}
\newcommand{\DeiM}[1]{\int_{\Dhat} #1 \, d\vec{r}}
\newcommand{\DeMd}[1]{\int_{\De,N_e} #1 \, d\vec{x}}
\newcommand{\DeiMd}[1]{\int_{\Dhat,N_e} #1 \, d\vec{r}}
\newcommand{\DeMdN}[1]{\int_{D,N} #1 \, d\vec{x}}
\newcommand{\DeiMdN}[1]{\int_{\Dhat,N} #1 \, d\vec{r}}
\newcommand{\DeiMdNC}[2]{\int_{\Dhat,N_{#2}} #1 \, d\vec{r}}
\newcommand{\pDei}[1]{\int_{\partial \De} #1 \, d\vec{x}}
\newcommand{\pDeiM}[1]{\int_{\partial \Dhat} #1 \, d\vec{r}}
\newcommand{\pDeiMd}[1]{\int_{\partial \Dhat,N_e} #1 \, d\vec{r}}
\newcommand{\pDeiMdN}[1]{\int_{\partial \Dhat,N} #1 \, d\vec{r}}
\newcommand{\pDeiMdNC}[2]{\int_{\partial \Dhat,N_{#2}} #1 \, d\vec{r}}
\newcommand{\pMortar}[2]{\int_{\mc{M}_{#2}} #1 \, d\vec{r}}
\newcommand{\pMortard}[2]{\int_{\mc{M}^{#2},N_{#2}} #1 \, d\vec{r}}
\newcommand{\pMortardd}[3]{\int_{\mc{M}^{#2},N_{#3}} #1 \, d\vec{r}}
\newcommand{\pMortardh}[3]{\int_{\mc{M}_{#2},N_{#3}} #1 \, d\vec{r}}
\newcommand{\pMortardhn}[4]{\int_{{#2}_{#3},N_{#4}} #1 \, d\vec{r}}
\newcommand{\TpDei}[1]{\int_{0}^T\int_{\partial \De} #1 \, d\vec{x}}
\newcommand{\TpDeid}[1]{\int_{0}^T\int_{\partial \De,N} #1 \, d\vec{x}}
\newcommand{\Deid}[1]{\int_{\De,N_e} #1 \, d\vec{x}}
\newcommand{\DeidN}[1]{\int_{\De,N} #1 \, d\vec{x}}
\newcommand{\pDeidN}[1]{\int_{\partial \De,N} #1 \, d\vec{x}}
\newcommand{\pDeid}[1]{\int_{\partial \De,N_e} #1 \, d\vec{x}}
\newcommand{\nor}[1]{\left\| #1 \right\|}
\newcommand{\snor}[1]{\left| #1 \right|}
\newcommand{\LRp}[1]{\left( #1 \right)}
\newcommand{\LRs}[1]{\left[ #1 \right]}
\newcommand{\LRa}[1]{\left< #1 \right>}
\newcommand{\LRc}[1]{\left\{ #1 \right\}}
\newcommand{\pp}[2]{\frac{\partial #1}{\partial #2}}
\newcommand{\ppcp}[1]{\frac{\partial #1}{\partial \vec{c_p}}}
\newcommand{\ppcpj}[1]{\frac{\partial #1}{\partial c_p}}
\newcommand{\ppcpmf}[1]{\frac{\partial #1}{\partial \vec{\mf{c_p}}}}
\newcommand{\ppm}[1]{\frac{\partial #1}{\partial \vec{m}}}
\newcommand{\ppho}[3]{\frac{\partial^{#3} #1}{{\partial #2}^{#3}}}
\newcommand{\ppmi}[1]{\frac{\partial #1}{\partial m_i}}
\newcommand{\ppmj}[1]{\frac{\partial #1}{\partial m_j}}
\newcommand{\ppcpm}[1]{\frac{\partial #1}{\partial \vec{c_p}^-}}
\newcommand{\ppcpp}[1]{\frac{\partial #1}{\partial \vec{c_p}^+}}
\newcommand{\ppcs}[1]{\frac{\partial #1}{\partial \vec{c_s}}}
\newcommand{\ppcsm}[1]{\frac{\partial #1}{\partial \vec{c_s}^-}}
\newcommand{\ppcsp}[1]{\frac{\partial #1}{\partial \vec{c_s}^+}}
\newcommand{\td}[2]{\frac{d #1}{d #2}}
\newcommand{\Rvert}[1]{\left. #1 \right|}
\newcommand{\bs}[1]{\boldsymbol{#1}}
\newcommand{\ipDed}[3]{\LRp{ #1, #2}_{\De,N}^{#3}}

\newcommand{\bigO}{\mathcal{O}}

\newcommand{\iOm}[1]{\int_{\Omega} #1 \, d\Omega}
\newcommand{\iGb}[1]{\int_{\Gamma_{b}} #1 \, ds}
\newcommand{\iGs}[1]{\int_{\Gamma_{s}} #1 \, ds}
\newcommand{\iGsy}[1]{\int_{\Gamma_{s}} #1 \, ds(\mb{y})}
\newcommand{\RE}{\mbox{{I}}\hspace{-0.5ex}\mbox{{R}}}
\newcommand{\iC}[1]{\int_{0}^{2\pi} #1 \, d\theta}
\newcommand{\iGr}[1]{\int_{\Gamma_{\infty}} #1 \, ds}

\newcommand{\jump}[1] {\ensuremath{[\![{#1}]\!]}}
\newcommand{\diff}[1] {\ensuremath{\LRs{#1}}}
\newcommand{\average}[1] {\ensuremath{\LRc{\!\!\LRc{#1}\!\!}}}

\newcounter{tablecell}

\newcommand*{\numbercell}{%
  \stepcounter{tablecell}%
  \thetablecell. 
}

\newtheorem{propo}{Proposition}
\newtheorem{coro}{Corollary}
\newtheorem{theo}{Theorem}
\newtheorem{lem}{Lemma}

\newtheorem{defi}{definition}

\newtheorem{rema}{remark}

\newcommand{\figlab}[1]{\label{fig:#1}}
\newcommand{\eqnlab}[1]{\label{eq:#1}}
\newcommand{\theolab}[1]{\label{theo:#1}}
\newcommand{\corolab}[1]{\label{coro:#1}}
\newcommand{\propolab}[1]{\label{propo:#1}}
\newcommand{\lemlab}[1]{\label{lem:#1}}
\newcommand{\defilab}[1]{\label{defi:#1}}
\newcommand{\remalab}[1]{\label{rema:#1}}
\newcommand{\tablab}[1]{\label{tab:#1}}

\newcommand{\figref}[1]{\ref{fig:#1}}
\newcommand{\theoref}[1]{\ref{theo:#1}}
\newcommand{\defiref}[1]{\ref{defi:#1}}
\newcommand{\remaref}[1]{\ref{rema:#1}}
\newcommand{\cororef}[1]{\ref{coro:#1}}
\newcommand{\proporef}[1]{\ref{propo:#1}}
\newcommand{\lemref}[1]{\ref{lem:#1}}
\newcommand{\eqnref}[1]{\eqref{eq:#1}}
\newcommand{\alglab}[1]{\label{alg:#1}}
\newcommand{\algref}[1]{\ref{alg:#1}}
\newcommand{\seclab}[1]{\label{sect:#1}}
\newcommand{\secref}[1]{\ref{sect:#1}}
\newcommand{\tabref}[1]{\ref{tab:#1}}

\renewcommand{\u}{u}
\newcommand{\uk}{\u^k}
\newcommand{\upk}{\u^{+k}}
\newcommand{\ukp}{\u^{k+1}}
\newcommand{\umkp}{\u^{-k+1}}
\newcommand{\umk}{\u^{-k}}
\renewcommand{\v}{v}
\newcommand{\vk}{{\v}^k}
\newcommand{\vkp}{{\v}^{k+1}}
\newcommand{\un}{\u_h}
\renewcommand{\e}{e}
\newcommand{\w}{w}
\renewcommand{\wb}{\mb{\w}}
\newcommand{\J}{J}
\newcommand{\Jb}{\mb{J}}
\newcommand{\q}{q}
\renewcommand{\r}{r}
\newcommand{\qh}{\hat{\q}}
\newcommand{\qs}{\q^*}
\newcommand{\qht}{\tilde{\qh}}
\newcommand{\qbar}{\overline{\q}}
\newcommand{\qb}{{\mb{\q}}}
\newcommand{\sig}{{\sigma}}
\newcommand{\thetah}{{\hat{\theta}}}
\newcommand{\thetas}{{\theta^*}}
\newcommand{\thetahn}{{\thetah_h}}
\newcommand{\sign}[1]{\text{ sgn}\!\LRp{#1}}
\newcommand{\sigh}{\hat{\sig}}
\newcommand{\sigs}{\sig^*}
\newcommand{\sigk}{\sig^{k}}
\newcommand{\sigpk}{\sig^{+k}}
\newcommand{\sigkp}{\sig^{k+1}}
\newcommand{\sigmkp}{\sig^{-k+1}}
\newcommand{\sigmk}{\sig^{-k}}
\newcommand{\sigb}{{\bs{\sig}}}
\newcommand{\sigbk}{\sigb^k}
\newcommand{\sigbkp}{\sigb^{k+1}}
\newcommand{\sigbn}{\sigb_h}
\newcommand{\sigbs}{\sigb^*}
\newcommand{\taub}{{\bs{\tau}}}
\newcommand{\taustar}{\tau_*}
\newcommand{\kappab}{{\bs{\kappa}}}
\newcommand{\gamb}{{\bs{\gamma}}}
\newcommand{\ut}{\tilde{\u}}
\newcommand{\sigbt}{\tilde{\sigb}}
\newcommand{\vt}{\tilde{\v}}
\newcommand{\taubt}{\tilde{\taub}}
\newcommand{\ub}{\mb{\u}}
\newcommand{\ubp}{\mb{\u}^+}
\newcommand{\ubm}{\mb{\u}^-}
\newcommand{\vb}{\mb{\v}}
\newcommand{\um}{\u^-}
\newcommand{\up}{\u^+}
\newcommand{\ubk}{\ub^k}
\newcommand{\ubkp}{\ub^{k+1}}

\newcommand{\m}{{m}}

\newcommand{\ud}{{\dot{\u}}}
\newcommand{\vd}{{\dot{\v}}}
\newcommand{\sigbd}{{\dot{\sigb}}}
\newcommand{\taubd}{\dot{\taub}}

\renewcommand{\d}{{d}}
\newcommand{\R}{{\mathbb{R}}}
\newcommand{\kap}{\kappa}
\newcommand{\F}{\mb{F}}
\newcommand{\f}{f}
\newcommand{\g}{g}
\newcommand{\fb}{\mb{\f}}
\newcommand{\gb}{\mb{\g}}
\newcommand{\gD}{g_D}
\newcommand{\pOmega}{\partial \Omega}
\newcommand{\K}{K}
\newcommand{\Kp}{K^+}
\newcommand{\Km}{K^-}
\newcommand{\pK}{{\partial \K}}
\newcommand{\pKin}{{\pK^{\text{in}}}}
\newcommand{\pKout}{{\pK^{\text{out}}}}
\newcommand{\Kj}{{\K_j}}
\newcommand{\Gh}{{\mc{E}_h}}
\newcommand{\Gho}{{\mc{E}_h^o}}
\newcommand{\Ghb}{{\mc{E}_h^{\partial}}}
\newcommand{\Poly}{{\mc{P}}}
\newcommand{\D}{{D}}
\newcommand{\p}{{p}}
\newcommand{\Ts}{{\mb{T}}}
\newcommand{\Vb}{{\mb{V}}}
\newcommand{\V}{{V}}
\newcommand{\Lam}{{\Lambda}}
\newcommand{\Lamb}{\bs{\Lambda}}
\newcommand{\mub}{\bs{\mu}}
\newcommand{\lambdab}{\bs{\lambda}}
\newcommand{\Th}{{\Ts_h}}
\newcommand{\Vh}{{\V_h}}
\newcommand{\Vbh}{{\Vb_h}}
\newcommand{\Lamh}{{\Lam_h}}
\newcommand{\Lambh}{{\Lamb_h}}
\newcommand{\Lamho}{{\overline{\Lam}_h}}
\newcommand{\Lambht}{{\Lamb_h^t}}
\newcommand{\Ltwo}{{L^2\LRp{\Omega}}}
\newcommand{\Ltwoe}{{L^2\LRp{\e}}}
\newcommand{\Ltw}{{L^2\LRp{\K}}}
\newcommand{\VhK}{{\V_h\LRp{\K}}}
\newcommand{\ThK}{{\Ts_h\LRp{\K}}}
\newcommand{\VbhK}{{\Vb_h\LRp{\K}}}
\newcommand{\Lamhe}{{\Lam_h\LRp{\e}}}
\newcommand{\Lambhe}{{\Lamb_h\LRp{\e}}}
\renewcommand{\L}{\mc{L}}
\renewcommand{\D}{\mc{D}}
\newcommand{\T}{\mc{T}}
\newcommand{\Tt}{\tilde{\T}}
\newcommand{\A}{\mb{A}}
\newcommand{\B}{\mb{B}}
\renewcommand{\D}{\mb{D}}
\newcommand{\Am}{\A^-}
\newcommand{\Ap}{\A^+}
\newcommand{\Aa}{\snor{\A}}
\newcommand{\Rb}{\mb{R}}
\newcommand{\C}{\mb{C}}
\renewcommand{\S}{\mb{S}}
\newcommand{\Sm}{\S^{-}}
\newcommand{\Sp}{\S^{+}}
\newcommand{\Spm}{\S^{\pm}}
\newcommand{\Q}{\mb{Q}}

\newcommand{\Lte}{{L^2\LRp{\Gh}}}
\newcommand{\n}{\mb{n}}
\newcommand{\nm}{\n^-}
\newcommand{\np}{\n^+}
\newcommand{\uh}{{\hat{\u}}}
\newcommand{\uhk}{{\uh}^k}
\newcommand{\uhkp}{{\uh}^{k+1}}
\newcommand{\us}{{\u^*}}
\newcommand{\uhn}{{\uh_h}}
\newcommand{\ubh}{{\hat{\ub}}}
\newcommand{\ubs}{{\ub^*}}
\newcommand{\ubht}{{\tilde{\ubh}}}
\newcommand{\uht}{{\tilde{\uh}}}
\newcommand{\uhtn}{{\uht_h}}
\newcommand{\uhd}{{\dot{\uh}}}
\newcommand{\ubhk}{\ubh^k}
\newcommand{\ubhkp}{\ubh^{k+1}}
\newcommand{\sigbh}{{\hat{\sigb}}}
\newcommand{\Fh}{{\hat{\F}}}
\newcommand{\Fs}{\F^*}
\renewcommand{\c}{{c}}

\newcommand{\zb}{\mb{z}}
\renewcommand{\sb}{\mb{s}}
\newcommand{\rb}{\mb{r}}
\newcommand{\betab}{\bs{\beta}}
\newcommand{\veps}{{\varepsilon}}
\newcommand{\vepsb}{\bs{\veps}}
\newcommand{\Hb}{\mb{H}}
\newcommand{\Hbt}{\Hb^t}
\newcommand{\Eb}{\mb{E}}
\newcommand{\Ebt}{\Eb^t}
\newcommand{\hb}{\mb{h}}
\newcommand{\eb}{\mb{e}}
\newcommand{\Hbh}{\hat{\mb{H}}}
\newcommand{\Ebh}{\hat{\mb{E}}}
\newcommand{\Hbs}{{\mb{H}^*}}
\newcommand{\Ebs}{{\mb{E}^*}}
\newcommand{\Ebht}{\Ebh^t}
\newcommand{\Hbht}{\Hbh^t}
\newcommand{\Lb}{\mb{L}}
\newcommand{\Gb}{\mb{G}}
\newcommand{\Lbh}{\hat{\Lb}}
\newcommand{\Lbs}{\Lb^*}
\newcommand{\Ib}{\mb{I}}
\newcommand{\Sigb}{\bs{\Sigma}}
\newcommand{\Piu}{\Pi_\u}
\newcommand{\Pisigb}{\Pi_\sigb}
\newcommand{\Z}{Z}
\newcommand{\Y}{Y}
\newcommand{\rhou}{\rho\u}
\newcommand{\rhov}{\rho\v}
\newcommand{\E}{E}
\renewcommand{\P}{P}
\newcommand{\h}{H}

\newcommand{\mygreen}[1]{{\color[rgb]{0,.65,0} #1}}
\newcommand{\mywhite}[1]{{\color[rgb]{1.0,1.0,1.0} #1}}
\newcommand{\myred}[1]{{\color[rgb]{0.65,0.0,0.0} #1}}
\newcommand{\myblue}[1]{{\color[rgb]{0.0,0.0,1.0} #1}}
\newcommand{\mygray}[1]{{\color[rgb]{.15,.35,.60} #1}}
\newcommand{\mythemec}[1]{{\color[RGB]{51,108,121} #1}}
\newcommand{\mydarkred}[1]{{\color[rgb]{.6,.1,.1} #1}}
\newcommand{\mydarkblue}[1]{{\color[rgb]{.1,.1,.9} #1}}
\newcommand{\mygreenback}[1]{{\color[rgb]{.75,1,.75} #1}}
\newcommand{\myorange}[1]{{\color[rgb]{.76,.39,.13} #1}}
\newcommand{\mygrass}[1]{{\color[rgb]{.19,.64,.13} #1}}
\newcommand{\mysierp}[1]{{\color[RGB]{209,28,209} #1}}

\newcommand{\tanbui}[2]{\textcolor{blue}{#1} \textcolor{red}{#2}}

\newcommand{\vel}{\bs{\vartheta}}
\newcommand{\vels}{\vel^*}
\newcommand{\vele}{\vel^e}
\newcommand{\velh}{\hat{\vel}}
\newcommand{\velk}{\vel^k}
\newcommand{\velkp}{\vel^{k+1}}
\newcommand{\vhk}{\hat\v^{k}}

\newcommand{\phis}{\phi^*}
\newcommand{\phie}{\phi^e}
\newcommand{\phih}{\hat{\phi}}
\newcommand{\phihk}{\phih^k}
\newcommand{\phik}{\phi^k}
\newcommand{\phikp}{\phi^{k+1}}

\begin{frontmatter}

\title{An Improved Iterative HDG Approach for Partial
   Differential Equations\tnoteref{t1}}
   \tnotetext[t1]{This research was
    partially supported by DOE grants DE-SC0010518 and
    DE-SC0011118, NSF Grants NSF-DMS1620352 and RNMS (Ki-Net) 1107444, ERC Advanced Grant DYCON: Dynamic Control. We are grateful for the supports.}

    \author[srk]{Sriramkrishnan Muralikrishnan}
    \author[Binh]{Minh-Binh Tran}
    \author[srk,Tan]{Tan Bui-Thanh}
    \address[srk]{Department of Aerospace Engineering and Engineering Mechanics,\\ The University of Texas at Austin, TX 78712, USA.} 
    \address[Tan]{The Institute for Computational Engineering
    \& Sciences,\\ The University of Texas at Austin, Austin, TX 78712,
    USA.}
    \address[Binh]{Department of Mathematics,  University of Wisconsin, Madison, WI 53706, USA.} 

\begin{abstract}
We propose and analyze an iterative high-order hybridized
discontinuous Galerkin (iHDG) discretization for linear partial
differential equations. We improve our previous work (SIAM
J. Sci. Comput. Vol. 39, No. 5, pp. S782--S808) in several directions:
1) the improved iHDG approach converges in a finite number of
iterations for the scalar transport equation; 2) it is unconditionally
convergent for both the linearized shallow water system and the
convection-diffusion equation; 3) it has improved stability and
convergence rates; 4) we uncover a relationship between the number of
iterations and time stepsize, solution order, meshsize and the
equation parameters. This allows us to choose the time stepsize such
that the number of iterations is approximately independent of the solution order and
the meshsize; and 5) we provide both strong and weak scalings of the
improved iHDG approach up to $16,384$ cores. A connection between
iHDG and time integration methods such as parareal and
implicit/explicit methods are discussed.  Extensive numerical
results are presented to verify the theoretical findings.

\end{abstract}


\end{frontmatter}


\section{Introduction}

Originally developed \cite{ReedHill73} for the neutron transport
equation, first analyzed in \cite{LeSaintRaviart74,
  JohnsonPitkaranta86}, the discontinuous Galerkin (DG) method has
been studied extensively for virtually all types of partial
differential equations (PDEs)
\cite{douglas1976interior,wheeler1978elliptic,arnold1982interior,
  CockburnKarniadakisShu00, arnold2002unified}. This is due to the fact that DG combines
advantages of finite volume and finite element methods. As such, it is
well-suited to problems with large gradients including shocks and with
complex geometries, and large-scale simulations demanding parallel
implementations.
In spite of these advantages, DG methods for
steady state and/or time-dependent problems that require implicit
time-integrators are more expensive in comparison to other existing
numerical methods since they typically have many more (coupled)
unknowns.

As an effort to mitigate the computational expense associated with DG
methods, the hybridized (also known as hybridizable) discontinuous
Galerkin (HDG) methods are introduced for various types of PDEs
including Poisson-type equation \cite{CockburnGopalakrishnanLazarov09,
  CockburnGopalakrishnanSayas10, KirbySherwinCockburn12,
  NguyenPeraireCockburn09a, CockburnDongGuzmanEtAl09,
  EggerSchoberl10}, Stokes equation \cite{CockburnGopalakrishnan09,
  NguyenPeraireCockburn10}, Euler and Navier-Stokes equations, wave
equations \cite{NguyenPeraireCockburn11, MoroNguyenPeraire11,
  NguyenPeraireCockburn11b, LiLanteriPerrrussel13,
  NguyenPeraireCockburn11a, GriesmaierMonk11, CuiZhang14}, to name a
few. In \cite{Bui-Thanh15, Bui-Thanh15a, bui2016construction}, we have
proposed an upwind HDG framework that provides a unified and a
systematic construction of HDG methods for a large class of PDEs.
We note that the weak Galerkin methods
in \cite{WangYe13, WangYe14, ZhaiZhangWang15, MuWangYe14}
share many similar advantages with HDG.

Besides the usual DG volume unknown, HDG methods introduce extra
single-valued trace unknowns on the mesh skeleton to reduce the number
of coupled degrees of freedom and to promote further parallelism. This
is accomplished via a Schur complement approach  in which the volume unknowns on each elements
are independently eliminated in parallel to provide a system of equations involving only the trace unknowns.
Moreover, the trace system is substantially smaller and
sparser compared to a standard DG linear system.
Once the trace unknowns are solved for, the volume unknowns can
be recovered in an element-by-element fashion, completely independent
of each other.  For small and medium sized problems the above approach
is popular. However, for practically large-scale applications,
where complex and high-fidelity simulations involving features with a
large range of spatial and temporal scales are necessary, the trace
system is still a bottleneck. In this case, matrix-free iterative
solvers/preconditioners \cite{Brown10,crivellini2011implicit,knoll2004jacobian} which converge in reasonably
small number of iterations are required.



Schwarz-type domain decomposition methods (DDMs) have become popular during the last three decades as they provide efficient algorithms to parallelize and to solve PDEs
\cite{Lions:1987:OSA,Lions:1989:OSA,Lions:1990:OSA}. Schwarz waveform relaxation methods and optimized Schwarz
methods
\cite{Halpern:OSM:2009,GanderGouarinHalpern:2011:OSW,GanderHajian:2015:ASM,Binh1,tran2011parallel,tran2013overlapping,tran2014behavior}
are among the most important subclasses of DDMs since they
can be adapted to the underlying physics, and thus lead to powerful
parallel solvers for challenging problems. While, scalable iterative solvers/preconditioners for the statically condensed
trace system can be developed \cite{CockburnDuboisGopalakrishnanEtAl13}, DDMs and HDG have a natural connection 
which can be exploited to create efficient parallel solvers. We have
developed and analyzed one such solver namely iterative HDG (iHDG) in
our previous work \cite{iHDG} for elliptic, scalar and system of
hyperbolic equations.  Independent and similar efforts for elliptic and
parabolic equations have been proposed and analyzed in
\cite{gander2014block,GanderHajian:2015:ASM,gander2016analysis}.

In the following, we discuss in section \secref{upwindhdg} an upwind
HDG framework \cite{Bui-Thanh15} for a general class of PDEs and also
our notations used in this paper.
The iHDG algorithm and significant
improvements over our previous work \cite{iHDG} are explained in section
\secref{iHDG}. The convergence of the new iHDG algorithm for the scalar and for system of hyperbolic PDEs is proved
in section \secref{iHDG-hyperbolic} using an energy approach.  In
section \secref{iHDG-convection-diffusion} we applied the iHDG approach for the convection-diffusion
PDE considered in the first order form. The convergence and
the scaling of the number of iHDG iterations with meshsize and solution order are derived.
Section \secref{numerical} presents various steady and time dependent
examples, in both two and three spatial dimensions, to support the
theoretical findings. We also present both strong and weak scaling results of our algorithm up to
16,384 cores in section \secref{numerical}. We finally conclude the paper in section
\secref{conclusions}.

\section{Upwind HDG framework}
\seclab{upwindhdg}
In this section we briefly review the upwind HDG framework for a
general system of linear PDEs and introduce necessary notations. To begin, let us consider the following
system of first order PDEs
\begin{equation}
\eqnlab{LPDE}
\sum_{k=1}^{\d}\partial_k\F_k\LRp{\ub} + \C\ub :=
\sum_{k=1}^{\d}\partial_k \LRp{\A_k\ub} + \C \ub = \fb, \quad \text{ in } \Omega,
\end{equation}
where $\d$ is the spatial dimension (which, for the clarity of the
exposition, is assumed to be $\d = 3$ whenever a particular value of
the dimension is of concern, but the result is also valid for $d =
\LRc{1,2}$), $\F_k$ the $k$th component of the flux vector (or tensor)
$\F$, $\ub$ the unknown solution with values in $\R^m$, and $\fb$ the
forcing term. For simplicity, we assume that the matrices $\A_k$ and $\C$ are 
continuous across $\Omega$. Here, $\partial_k$
stands for the $k$-th partial derivative and by the subscript $k$ we
denote the $k$th component of a vector/tensor. We shall employ
HDG to discretize 
\eqnref{LPDE}. To that end, let us start with the
notations and conventions.

We partition $\Omega \in \R^\d$, an  open and bounded domain,  into
$\Nel$ non-overlapping elements $\Kj, j = 1, \hdots, \Nel$ with
Lipschitz boundaries such that $\Omega_h := \cup_{j=1}^\Nel \Kj$ and
$\overline{\Omega} = \overline{\Omega}_h$. The mesh size $h$ is defined as $h
:= \max_{j\in \LRc{1,\hdots,\Nel}}diam\LRp{\Kj}$. We denote the
skeleton of the mesh by $\Gh := \cup_{j=1}^\Nel \partial K_j$,
the set of all (uniquely defined) faces $\e$. We conventionally identify $\nm$ as the outward 
normal vector on the boundary $\pK$ of element $\K$ (also denoted as $\Km$) and $\np = -\nm$ as the outward normal vector of the boundary of a neighboring element (also denoted as $\Kp$). Furthermore, we use $\n$ to denote either $\nm$ or $\np$ in an expression that is valid for both cases, and this convention is also used for other quantities (restricted) on
a face $\e \in \Gh$.
   For  convenience, we denote by $\Ghb$ the sets of all boundary faces on $\pOmega$, by $\Gho := \Gh \setminus \Ghb$ the set of all interior faces, and $\pOmega_h := \LRc{\pK:\K \in \Omega_h}$.


For simplicity in writing we define $\LRp{\cdot,\cdot}_\K$ as the
$L^2$-inner product on a domain $\K \in \R^\d$ and
$\LRa{\cdot,\cdot}_\K$ as the $L^2$-inner product on a domain $\K$ if
$\K \in \R^{\d-1}$. We shall use $\nor{\cdot}_{\K} :=
\nor{\cdot}_{\Ltw}$ as the induced norm for both cases and the
particular value of $\K$ in a context will indicate the inner
product from which the norm is coming. We also denote the $\veps$-weighted
norm of a function $\u$ as $\nor{\u}_{\veps, \K} :=
\nor{\sqrt{\veps}\u}_{\K}$ for any positive $\veps$.  We shall use
boldface lowercase letters for vector-valued functions and in that
case the inner product is defined as $\LRp{\ub,\vb}_\K :=
\sum_{i=1}^m\LRp{\ub_i,\vb_i}_\K$, and similarly $\LRa{\ub,\vb}_\K :=
\sum_{i = 1}^m\LRa{\ub_i,\vb_i}_\K$, where $\m$ is the number of
components ($\ub_i, i=1,\hdots,\m$) of $\ub$.  Moreover, we define
$\LRp{\ub,\vb}_\Omega := \sum_{\K\in \Omega_h}\LRp{\ub,\vb}_\K$ and
$\LRa{\ub,\vb}_\Gh := \sum_{\e\in \Gh}\LRa{\ub,\vb}_\e$ whose
induced (weighted) norms are clear, and hence their definitions are
omitted. We  employ boldface uppercase letters, e.g. $\mb{L}$, to
denote matrices and tensors. 
We conventionally use $\ub$ ($\vb$ and $\ubh)$ for
the numerical solution and $\ub^e$ for the exact solution.

We denote by $\Poly^\p\LRp{\K}$ the space of polynomials of degree at
most $\p$ on a domain $\K$. Next, we introduce two discontinuous
piecewise polynomial spaces
\begin{align*}
\Vbh\LRp{\Omega_h} &:= \LRc{\vb \in \LRs{L^2\LRp{\Omega_h}}^m:
  \eval{\vb}_{\K} \in \LRs{\Poly^\p\LRp{\K}}^m, \forall \K \in \Omega_h}, \\
\Lambh\LRp{\Gh} &:= \LRc{\lambdab \in \LRs{\Lte}^m:
  \eval{\lambdab}_{\e} \in \LRs{\Poly^\p\LRp{\e}}^m, \forall \e \in \Gh},
\end{align*}
and similar spaces  $\VbhK$ and $\Lambhe$ on $\K$ and $\e$ by replacing $\Omega_h$ with
$\K$ and $\Gh$ with $\e$. For scalar-valued functions,
we denote the corresponding spaces as
\begin{align*}
\Vh\LRp{\Omega_h} &:= \LRc{\v \in L^2\LRp{\Omega_h}:
  \eval{\v}_{\K} \in \Poly^\p\LRp{\K}, \forall \K \in \Omega_h}, \\
\Lamh\LRp{\Gh} &:= \LRc{\lambda \in \Lte:
  \eval{\lambda}_{\e} \in \Poly^\p\LRp{\e}, \forall \e \in \Gh}.
\end{align*}


Following \cite{Bui-Thanh15}, an upwind HDG discretization for \eqnref{LPDE} 
in each element $\K$ involves
the DG local unknown $\ub$ and the extra ``trace'' unknown $\ubh$ such that
\begin{subequations}
\eqnlab{KHDG}
\begin{align}
\eqnlab{HDGlocal}
-\LRp{\F\LRp{\ub}, \Grad\vb}_\K + \LRa{\Fh\LRp{\ub,\ubh}\cdot
\n,\vb}_\pK + \LRp{\C\ub,\vb}_\K &= \LRp{\fb,\vb}_\K, \\
\eqnlab{HDGeqn}
\LRa{\jump{\Fh\LRp{\ub,\ubh} \cdot \n},\mub}_\e &= \mb{0}, \quad \forall \e \in \Gho, 
\end{align}
\end{subequations}
 where we have defined the ``jump'' operator for any quantity
 $\LRp{\cdot}$ as $\jump{\LRp{\cdot}} := \LRp{\cdot}^- +
 \LRp{\cdot}^+$. We also define the ``average'' operator
 $\average{\LRp{\cdot}}$ via $2\average{\LRp{\cdot}} :=
 \jump{\LRp{\cdot}}$. For simplicity, we have ignored the fact that equations \eqnref{HDGlocal}, \eqnref{HDGeqn} 
 must hold for
 all test functions $\vb \in \VbhK$ and $\mub \in \Lambh\LRp{\e}$ respectively. 
 This is implicitly understood throughout the paper. Here, the HDG flux is defined as
\begin{equation}
\eqnlab{RiemannFluxBoth}
\Fh\cdot \n = \F\LRp{\ub}\cdot \n + \Aa \LRp{\ub - \ubh},
\end{equation}
with the matrix $\A := \sum_{k = 1}^\d{\A_k\n_k} = \Rb \S \Rb^{-1}$, and
  $\Aa := \Rb \snor{\S}\Rb^{-1}$. Here $\n_k$ is the $k$th component of the outward normal vector 
  $\n$ and $\snor{\S}$ represents a matrix obtained by taking the absolute 
  value of the main diagonal of the matrix $\S$. We have  assumed that $\A$ admits an eigen-decomposition, and
  this is valid for a large class of PDEs of Friedrichs' type \cite{Friedrichs58}. 

  The typical procedure for computing HDG solution requires three
  steps. We first solve \eqnref{HDGlocal} for the local solution $\ub$
  as a function of $\ubh$. It is then substituted into the
  conservative algebraic equation \eqnref{HDGeqn} on the mesh skeleton
  to solve for the unknown $\ubh$. Finally, the local unknown $\ub$ is
  computed, as in the first step, using $\ubh$ from the second
  step. Since the number of trace unknowns $\ubh$ is less than the DG
  unknowns $\ub$ \cite{bui2016construction}, HDG is more advantageous.
  For large-scale problems, however, the trace system on the mesh
  skeleton could be large and iterative solvers are necessary. In the
  following we construct an iterative solver that takes advantage of
  the HDG structure and the domain decomposition method. 

\section{iHDG methods}
\seclab{iHDG} 

To reduce the cost of solving the trace system, our previous effort
\cite{iHDG} is to break the coupling between $\ubh$ and $\ub$ in
\eqnref{KHDG} by iteratively solving for $\ub$ in terms of $\ubh$ in
\eqnref{HDGlocal}, and $\ubh$ in terms of $\ub$ in \eqnref{HDGeqn}. We
name this approach  iterative HDG (iHDG) method, and now let us call it iHDG-I to
distinguish it from the approach developed in this paper. From a linear algebra
point of view, iHDG-I can be considered as a block Gauss-Seidel approach for the
system \eqnref{KHDG} that requires only independent element-by-element
and face-by-face local solves in each iteration. However, unlike
conventional Gauss-Seidel schemes which are purely algebraic, the
convergence of iHDG-I \cite{iHDG} does not depend on the ordering of
unknowns. From the domain decomposition point of view, thanks to the
HDG flux, iHDG can be identified as an optimal Schwarz method in which
each element is a subdomain. Using an energy approach, we have
rigorously shown the convergence of the iHDG-I for the transport
equation, the linearized shallow water equation and the
convection-diffusion equation \cite{iHDG}.  

Nevertheless, {\em a number of questions that need to be addressed for the
iHDG-I approach}. First, with the upwind flux it theoretically takes
infinite number of iterations to converge for the scalar transport
equation. Second, it is conditionally convergent for the linearized
shallow water system; in particular, it blows up for fine meshes and/or
large time stepsizes. Furthermore, we have not been able to estimate
the number of iterations as a function of time stepsize, solution
order, and  meshsize. Third, it is also conditionally convergent
for the convection-diffusion equation, especially in the diffusion-dominated regime. 

The iHDG approach constructed in this paper, which we call iHDG-II,
overcomes all the aforementioned shortcomings. In particular, it
converges in a finite number of iterations for the scalar transport
equation and is unconditionally convergent for both the linearized shallow
water system and the convection-diffusion equation. Moreover, compared to
our previous work \cite{iHDG}, we provide several additional findings:
1) we make a connection between iHDG and the parareal method, which reveals interesting similarities 
 and differences between the two methods;
2) we show that iHDG can be considered as a {\em locally implicit}
method, and hence being somewhat in between fully explicit and fully
implicit approaches; 3) for both the linearized shallow water system and the
convection-diffusion equation, using an asymptotic approximation, we
uncover a relationship between the number of iterations and time
stepsize, solution order, meshsize and the equation parameters. This allows us to choose the
time stepsize such that the number of iterations is approximately independent of the
solution order and the meshsize; 4) we show that iHDG-II has improved stability and
convergence rates over iHDG-I; and 5) we provide both strong and weak scalings
of the iHDG-II approach up to $16,384$ cores.

We now present a detailed construction of the iHDG-II approach.
We define the approximate  solution for the volume variables 
at the $\LRp{k+1}$th iteration using the local equation \eqnref{HDGlocal} as
\begin{align}
\nonumber
-\LRp{\F\LRp{\ubkp}, \Grad\vb}_\K + \LRa{\F\LRp{\ubkp}\cdot \n + \Aa(\ubkp-\ubh^{k,k+1}),\vb}_\pK \\ 
\eqnlab{iHDGlocal}
+ \LRp{\C\ubkp,\vb}_\K  = \LRp{\fb,\vb}_\K,
\end{align}
where the weighted trace $\Aa\ubh^{k,k+1}$ is computed from \eqnref{HDGeqn} using volume unknown
in element $\K$ at the $\LRp{k+1}$th iteration, i.e. $\LRp{\ubkp}^-$,
and volume solution of the neighbors at  the $\LRp{k}$th
iteration, i.e. $\LRp{\ubk}^+$:
\begin{align}
\nonumber
\LRa{2\Aa\ubh^{k,k+1},\mub}_\pK = &\LRa{\Aa\LRc{\LRp{\ubkp}^-+\LRp{\ubk}^+},\mub}_\pK \\ 
\eqnlab{iHDGtrace}
                          & +\LRa{\F\LRc{\LRp{\ubkp}^-}\cdot \n^{-} +\F\LRc{\LRp{\ubk}^+}\cdot \n^{+}, \mub}_\pK.
\end{align}

Algorithm \ref{al:DDMhyperbolic} summarizes the iHDG-II
approach. Compared to iHDG-I, iHDG-II improves the coupling between
$\ubh$ and $\ub$ while still avoiding intra-iteration communication
between elements. The trace $\ubh$ is double-valued during the course
of iterations for iHDG-II and in the event of convergence it becomes
single valued upto a specified tolerance. Another principal difference
is that while the well-posedness of iHDG-I elemental local solves is
inherited from the original HDG counterpart, {\em it has to be shown
  for iHDG-II}.  This is due to the new way of computing the weighted
trace in \eqnref{iHDGtrace} that involves $\ubkp$, and hence changing
the structure of the local solves.  Similar and independent work for
HDG methods for elliptic/parabolic problems have appeared in
\cite{gander2014block,GanderHajian:2015:ASM,gander2016analysis}. Here,
we are interested in pure hyperbolic equations/systems and
convection-diffusion equations.  Unlike existing matrix-based
approaches, our convergence analysis is based on an energy approach
that exploits the variational structure of HDG methods. Moreover we
provide, both rigorous and asymptotic, relationships between the
number of iterations and time stepsize, solution order, meshsize and
the equation parameters. We also make connection between our proposed iHDG-II approach
with parareal and time integration methods. Last but not least, our
 framework is more general: indeed it recovers the contraction factor
results in \cite{gander2014block} for elliptic
equations as one of the special cases.

\begin{algorithm}
  \begin{algorithmic}[1]
    \ENSURE Given initial guess $\ub^0$, compute the weighted trace $\Aa\ubh^{0,1}$ using \eqnref{iHDGtrace}.
      \vspace{-4mm}
      \WHILE{not converged} 
      \STATE Solve the local equation \eqnref{iHDGlocal} for $\ubkp$ using the weighted trace $\Aa\ubh^{k,k+1}$.
      \vspace{-4mm}
      \STATE Compute $\Aa\ubh^{k+1,k+2}$ using \eqnref{iHDGtrace}.
      \STATE Check convergence. If yes, {\bf exit}, otherwise {\bf set} $k = k+ 1$ and {\bf continue}.
    \ENDWHILE
  \end{algorithmic}
  \caption{The iHDG-II approach.}
  \label{al:DDMhyperbolic}
\end{algorithm}

\section{iHDG-II  for  hyperbolic PDEs}
\seclab{iHDG-hyperbolic} 

In this section we show that iHDG-II improves upon iHDG-I in many
aspects discussed in section \secref{iHDG}. The PDEs of interest are
(steady and time dependent) transport equation, and the linearized shallow water system \cite{iHDG}.


\subsection{Transport equation}
Let us start with the (steady) transport equation
\begin{subequations}
\eqnlab{transport}
\begin{align}
\betab \cdot \Grad \u^e &= f \quad \text{ in }
\Omega, \\
\u^e &= \g \quad \text{ on } \pOmega^-,
\end{align}
\end{subequations}
where $\pOmega^-$ is the inflow part of the boundary $\pOmega$, and
again $\u^e$ denotes the exact solution. {\it Note that $\betab$ is
  assumed to be continuous across the mesh skeleton}. 

Applying the iHDG-II algorithm \ref{al:DDMhyperbolic} to the upwind HDG
discretization \cite{iHDG} for \eqnref{transport} we obtain the
approximate solution $\u^{k+1}$ at the $(k+1)$th iteration restricted
on each element $\K$ via the following independent local solve:
\begin{multline}
    -\LRp{\LRp{\ukp}^-,\Div\LRp{\betab \v}}_\K  + \LRa{\betab\cdot\n^-\LRp{\ukp}^- + \snor{\betab\cdot \n}\LRc{\LRp{\ukp}^- -\uh^{k,k+1}},\v}_\pK \\ 
    =\LRp{f,\v}_\K,
\eqnlab{transportLocalk1}
\end{multline}
       where the weighted trace  $\snor{\betab\cdot \n}\uh^{k,k+1}$ is computed using information from the previous iteration and current iteration as
\begin{multline}
\eqnlab{transportLocale}
2\snor{\betab\cdot \n}\uh^{k,k+1} = \LRc{\betab\cdot \n^-\LRp{\ukp}^- + \betab\cdot\n^+\LRp{\uk}^+} \\ 
+\snor{\betab\cdot \n}\LRc{\LRp{\ukp}^-+\LRp{\uk}^+}.
\end{multline}

Next we study the convergence of the iHDG-II method
\eqnref{transportLocalk1}, \eqnref{transportLocale}. Since
\eqnref{transport} is linear, it is sufficient to show that the algorithm
converges to the zero solution for the homogeneous equation with zero forcing $\f=0$ and zero
boundary condition $\g=0$. Let us  define $\pKout$ as the outflow part of $\pK$,
i.e. $\betab\cdot \n^{-} > 0$
on $\pKout$, and $\pKin$ as the inflow
part of $\pK$, i.e. $\betab\cdot \n^{-}< 0$
on $\pKin$. First, we will prove the well-posedness of the local solver \eqnref{transportLocalk1}.
\begin{lemma}
    \lemlab{wellposednesstransport}
    Assume $-\Div \betab \ge \alpha > 0$, i.e. \eqnref{transport} is well-posed. Then the local solver \eqnref{transportLocalk1} of the iHDG-II algorithm for the transport equation is well-posed. 
\end{lemma}
\begin{proof}
Taking $\v = \LRp{\ukp}^-$ in \eqnref{transportLocalk1}, substituting \eqnref{transportLocale} in \eqnref{transportLocalk1} and applying homogeneous forcing condition yield

\begin{multline}
\eqnlab{transportLocalb}
-\LRp{\LRp{\ukp}^-,\Div\LRc{\betab \LRp{\ukp}^-}}_\K  + \frac{1}{2}\LRa{\LRp{\betab\cdot\n^{-}+\snor{\betab\cdot \n}}\LRp{\ukp}^-,\LRp{\ukp}^-}_\pK \\ 
=\frac{1}{2}\LRa{\LRp{\betab\cdot\n^{+}+\snor{\betab\cdot \n}}\LRp{\uk}^+,\LRp{\ukp}^-}_\pK.
\end{multline}
 Since
 \begin{multline}\nonumber
     \LRp{\LRp{\ukp}^-,\nabla\cdot\LRc{\betab \LRp{\ukp}^-}}_{K}=\LRp{\LRp{\ukp}^-,\Div\betab \LRp{\ukp}^{-}}_{K}\\
     \nonumber
                                                                 +\LRp{\LRp{\ukp}^-,\betab \cdot \nabla  \LRp{\ukp}^-}_{K},
 \end{multline}
integrating by parts the second term on the right hand side
\begin{multline}\nonumber
    \LRp{\LRp{\ukp}^-,\nabla\cdot\LRc{\betab \LRp{\ukp}^-}}_{K} =\LRp{\LRp{\ukp}^-,\Div \betab \LRp{\ukp}^-}_{K}\\
    \nonumber
    -\LRp{\LRp{\ukp}^-,\nabla\cdot\LRc{\betab \LRp{\ukp}^-}}_{K} +\LRa{\betab\cdot \n^-\LRp{\ukp}^-, \LRp{\ukp}^-}_{\partial K},\nonumber
 \end{multline}
yields the following identity, after rearranging the terms
\begin{multline}
\eqnlab{Identity}
\LRp{\LRp{\ukp}^-,\nabla\cdot\LRc{\betab \LRp{\ukp}^-}}_{\K} 
=\LRp{\LRp{\ukp}^-,\frac{\Div \betab}{2} \LRp{\ukp}^-}_{\K} \\
 +\frac{1}{2}\LRa{\betab\cdot \n^-\LRp{\ukp}^- ,\LRp{\ukp}^-}_{\pK}.
\end{multline}
Using \eqnref{Identity} in \eqnref{transportLocalb} we get
\begin{multline}
\label{wellposednesstransport}
    \nor{\LRp{\ukp}^-}^2_{\frac{-\Div \betab}{2}, \K}+\nor{\LRp{\ukp}^-}^2_{|\betab\cdot\n|/2,\pK}= \\
\frac{1}{2}\LRa{\LRp{\betab\cdot\n^{+}+\snor{\betab\cdot \n}}\LRp{\uk}^+,\LRp{\ukp}^-}_\pK.
\end{multline}
In equation \eqref{wellposednesstransport} all the terms on the left hand side are positive. Since $\LRp{\uk}^+$ is the ``forcing'' for the local equation, by taking $\LRp{\uk}^+=0$ the only solution possible is $\LRp{\ukp}^-=0$ and hence the local solver is well-posed.
\end{proof}
Having proved the well-posedness of the local solver we can now proceed to prove the convergence of algorithm 
\ref{al:DDMhyperbolic} for the transport equation.
\begin{theorem}
\theolab{DDMConvergence} Assume $-\Div \betab \ge \alpha > 0$,
i.e. \eqnref{transport} is well-posed.  There exists $J \le \Nel$ such
that the iHDG-II algorithm for the homogeneous
transport equation converges to the HDG solution in $J$ iterations.
\end{theorem}

\begin{proof}
    Using \eqref{wellposednesstransport} from Lemma \lemref{wellposednesstransport} and $\betab\cdot \n^{+}> 0$ on $\pKin$, $\betab\cdot \n^{+}< 0$ on $\pKout$ we can write
\begin{equation}
\label{transportLocalc}
\nor{\LRp{\ukp}^-}^2_{\frac{-\Div \betab}{2}, \K}+\nor{\LRp{\ukp}^-}^2_{|\betab\cdot\n|/2,\pK}= \LRa{|\betab\cdot \n|{\u}^k_{\text{ext}}, \LRp{\ukp}^-}_{\pKin}.
\end{equation}
where ${\u}^k_{\text{ext}}$ is either the physical boundary condition or the solution of the neighboring
element that shares the same inflow boundary $\pKin$.

Consider the set $\mathcal{K}^1$ of all elements $K$ such that
$\pKin$ is a subset of the physical inflow boundary
$\pOmega^{\text{in}}$ on which
we have $\u^k_{\text{ext}}=0$ for all $k \in \mathbb{N}$.
 We obtain from \eqref{transportLocalc} that
\begin{equation}
\eqnlab{transportLocalh}
\nor{\LRp{\ukp}^-}^2_{\frac{-\Div \betab}{2}, \K}+\nor{\LRp{\ukp}^-}_{|\betab\cdot \n|/2,\pK}^2
= 0,
\end{equation}
which implies $\u^{1}=0$ on $\K \in \mc{K}^1$, i.e. our iterative solver is exact on $K\in \mathcal{K}^1$ at the first iteration.

Next, let us define $\Omega^1_h:=\Omega_h$ and  
\[
\Omega^2_h=\Omega^1_h\backslash \mc{K}^1.
\]
Consider the set $\mathcal{K}^2$ of all $K$ in $\Omega^2_h$ such that $\pKin$ is either
(possibly partially) a subset of the physical inflow boundary
$\pOmega^{\text{in}}$ or (possibly partially) a subset of the outflow
boundary of elements in $\mc{K}^1$. This implies, on $\pKin \in \mc{K}^2$, $\u^k_{\text{ext}}=0$
for all $k \in \mathbb{N}\setminus\LRc{1}$. Thus, $\forall \K \in \mc{K}^2$, we have
\begin{equation}\label{NtransportLocalk}
    \nor{\LRp{\uk}^-}_{\frac{-\Div \betab}{2},\K}^2 +
    \nor{\LRp{\uk}^-}_{|\betab\cdot \n|/2,\pK}^2 = 0, \quad \forall k \in \mathbb{N}\setminus\LRc{1},
\end{equation}
which implies $\u^{2}=0$ in $\K \in \mc{K}^2$, i.e. our iterative solver is exact on $K \in \mathcal{K}^2$ at the second iteration.

 Repeating the same argument, we can construct subsets $\mathcal{K}^j
 \subset \Omega_h$, on which the iterative solution on $K\in
 \mathcal{K}^j$ is the exact HDG solution at the $j$-th iteration. Since
 the number of elements $\Nel$ is finite, there exists $J \le \Nel$ such that
 $\Omega_h = \cup_{j=1}^J \mc{K}^j$. It follows that the iHDG-II algorithm provides exact HDG solution on $\Omega_h$
 after $J$ iterations.
\end{proof}

\begin{remark}
  \remalab{matching}
Compared to iHDG-I \cite{iHDG}, which requires an infinite number of
iterations to converge, iHDG-II needs finite number of iterations for convergence. The key to the
improvement is the stronger coupling between $\ubh$ and $\ub$ by using
$\LRp{\ubkp}^-$ in \eqnref{iHDGtrace} instead of $\LRp{\ubk}^-$. The
proof of Theorem \theoref{DDMConvergence} also shows that iHDG-II
automatically marches the flow, i.e., each iteration yields the
HDG solution exactly for a group of elements. Moreover, the marching process is
automatic (i.e. does not require an ordering of elements) and adapts
to the velocity field $\betab$ under consideration.
\end{remark}

\subsection{Time-dependent transport equation}
\seclab{iHDG-timedependent}
In this section we first comment on a space-time formulation of the iHDG methods and compare it with the parareal methods studied in \cite{gander2008analysis} for the time-dependent scalar transport equation. Then we consider the semi-discrete version of iHDG combined with traditional time integration schemes and compare it with
the fully implicit and explicit DG/HDG schemes.

\subsubsection{Comparison of space-time iHDG and parareal methods for the scalar transport equation}
\seclab{iHDG-parareal}


Space-time finite element methods have been studied extensively for the past several years both in the context of continuous and discontinuous Galerkin methods \cite{kaczkowski1975method,argyris1969finite,oden1969general,klaij2006space,ellis2016robust} and  HDG methods \cite{rhebergen2012space}. Parareal methods, on the other hand, were first introduced in \cite{lions2001resolution}  and various modifications have been  proposed and studied (see \cite{garrido2006convergent,farhat2003time,maday2002parareal,minion2011hybrid,gander2007analysis} and references therein).

In the scope of our work, we compare our methods with the parareal scheme proposed in \cite{gander2008analysis} for the scalar advection equation. Let us start with the following ordinary differential equation 
\begin{equation}\label{ODE}
\DD{u}{t} = f \quad \text{ in }
    (0,T), \ \
u(0) = \g,
\end{equation}
for some positive constant $T>0$. 
\begin{corollary}
\corolab{ODE_iHDG}
Suppose we discretize the temporal derivative in \eqref{ODE} using the
iHDG-II method with the upwind flux and the elements $\K_j$ are
ordered such that $\K_j$ is on the left of $\K_{j+1}$. At iteration
$k$, $\eval{\u^k}_{\K_j}$ converges to the HDG solution
$\eval{\u}_{\K_j}$ for $j \le k$.
\end{corollary}
\begin{proof}
    Since \eqref{ODE} can be considered as 1D transport equation \eqnref{transport} with  velocity $\betab=1$, the proof follows directly from Theorem \theoref{DDMConvergence} and induction.  
\end{proof}

Note that the iHDG scheme can be considered as a parareal algorithm in
which the fine propagator is taken to be the local solver
\eqnref{iHDGlocal} and the coarse propagator corresponds to the
conservation condition \eqnref{iHDGtrace}. However, unlike existing
parareal algorithms, the coarse propagator of iHDG-parareal is dependent on the
fine propagator. Moreover, Corollary \cororef{ODE_iHDG}
says that after $k$ iterations the iHDG-parareal solution converges up
to element $k$, a feature common to the parareal algorithm studied in \cite{gander2008analysis}. For time dependent hyperbolic PDEs, the space-time
iHDG method again can be understood as parareal approach, and in this
case, a layer of space-time elements converges after each
iHDG-parareal iteration (see Remark \remaref{matching}). See Figure \figref{shocksoln} and Table
\ref{tab:2d_discont_3d_smooth} of section \secref{numerical} for a
demonstration in 2D where either $x$ or $y$ is considered as ``time''.
It should be pointed out that the specific parareal method in
\cite{gander2008analysis} exactly traces the characteristics, and
hence may take less iterations to converge than the iHDG-parareal
method, but this is only true if the forward Euler discretization in
time, upwind finite difference in space, and $CFL=1$ are used with
constant advection velocity.

\subsection{iHDG as a locally implicit method}

In this section we discuss the relationship between iHDG and
implicit/explicit HDG methods. For the simplicity of the exposition,
we consider time-dependent scalar transport equation given by:
\begin{equation}
\eqnlab{time_transport}
    \frac{\partial{\u^e}}{\partial{t}} + \betab \cdot \Grad\u^e = f. 
\end{equation}
We first review the implicit/explicit HDG schemes for
\eqnref{time_transport}, and then compare them with iHDG-II. The implicit Euler HDG scheme for \eqnref{time_transport} reads
\begin{multline}
\eqnlab{implicitHDG}
    \LRp{\frac{\u^{m+1}}{\Delta t}, \v}_\K-\LRp{\u^{m+1}, \Div(\betab\v)}_\K + \LRa{\betab \cdot \n \u^{m+1} + \tau(\u^{m+1} - \uh^{m+1}),\v}_\pK  \\
=\LRp{\f^{m+1}+\frac{\u^m}{\Delta t},\v}_\K,\\
\LRa{\jump{\tau\uh^{m+1}},\mub}_\pK = \LRa{\jump{\tau\u^{m+1}} + \jump{\betab \cdot \n \u^{m+1}}, \mub}_\pK.
  \end{multline}
Here, $\u^{m+1}$ and $\uh^{m+1}$ stands for 
the volume and the trace unknowns at the current 
time step, whereas $\u^{m}$ and $\uh^{m}$ are the computed solutions from the previous 
time step. Clearly,  $\u^{m+1}$ and $\uh^{m+1}$ are coupled and this can
be a challenge for large-scale problems.


Next let us consider an explicit HDG with forward Euler discretization in time for \eqnref{time_transport}:
\begin{equation*}
    \LRp{\frac{\u^{m+1}}{\Delta t}, \v}_\K = \LRp{\u^{m}, \Div(\betab\v)}_\K - \LRa{\betab \cdot \n \u^{m} + \tau(\u^{m} - \uh^{m}),\v}_\pK + \LRp{\f^{m}+\frac{\u^m}{\Delta t},\v}_\K,
\end{equation*}
\begin{equation*}
\LRa{\jump{\tau\uh^{m}},\mub}_\pK = \LRa{\jump{\tau\u^{m}} + \jump{\betab \cdot \n \u^{m}}, \mub}_\pK,
\end{equation*}
which shows that we can solve for $\u^{m+1}$ element-by-element, completely independent of each other.
However, since it is an explicit scheme, the $CFL$ restriction for stability can increase the 
computational cost for problems involving fast time scales and/or fine meshes.


Now applying {\em one iteration} of the iHDG-II scheme
for the implicit HDG formulation \eqnref{implicitHDG} 
with $\u^m$ as the initial guess yields
\begin{multline}
\nonumber
    \LRp{\frac{\u^{m+1}}{\Delta t}, \v}_\K-\LRp{\u^{m+1}, \Div(\betab\v)}_\K + \LRa{\betab \cdot \n \u^{m+1} + \tau(\u^{m+1}-\uh^{m,m+1}),\v}_\pK  = \\
    \LRp{\f^{m+1}+\frac{\u^m}{\Delta t},\v}_\K,\\
    \LRa{\jump{\tau\uh^{m,m+1}},\mub}_\pK = \LRa{\tau^-\LRp{\u^{m+1}}^-+\tau^+\LRp{\u^m}^+, \mub}_\pK \\
                                            +\LRa{\betab \cdot \n^-\LRp{\u^{m+1}}^-+\betab \cdot \n^+\LRp{\u^m}^+, \mub}_\pK.
\end{multline}
Compared to the explicit HDG scheme, iHDG-II requires local solves
since it is {\em locally implicit}. As such, its CFL restriction is
much less (see Figure \ref{CFL_vs_k}), while still having similar
parallel scalability of the explicit method.\footnote{In fact, the
  parallel efficiency could be much more than explicit methods since
  the local solves can be well overlapped with the communication.}
Indeed, Figure \ref{CFL_vs_k} shows that the CFL restriction is only
indirectly through the increase of the number of iterations; for CFL
numbers between $1$ and $5$, the number of iterations varies
mildly. Thus, as a locally implicit method, iHDG-II combines
advantages of both explicit (e.g. matrix free and parallel
scalability) and implicit (taking reasonably large time stepsize
without facing instability) methods. Clearly, on convergence iHDG
solution is, up to the stopping tolerance, the same as the
fully-implicit solution.


\subsection{iHDG-II for system of hyperbolic PDEs}
\label{iiHDG_shallow}
In this section, as an example for the system of linear hyperbolic PDEs, we
consider the following linearized oceanic  shallow water system
\cite{GiraldoWarburton08}:
\begin{equation}
    \pp{}{t}\LRp{
      \begin{array}{c}
        \phi^e \\
       \Phi \u^e \\
        \Phi \v^e
      \end{array}
    } + 
    \pp{}{x}\LRp{
      \begin{array}{c}
        \Phi  \u^e \\
        \Phi\phi^e \\
        0
      \end{array}
    } + 
    \pp{}{y}\LRp{
      \begin{array}{c}
        \Phi \v^e \\
        0 \\
        \Phi\phi^e
      \end{array}
    } = 
    \LRp{
      \begin{array}{c}
        0 \\
        f\Phi\v^e - \gamma \Phi\u^e + \frac{\tau_x}{\rho} \\
       -f\Phi\u^e - \gamma \Phi\v^e + \frac{\tau_y}{\rho}
      \end{array}
    }
\eqnlab{linearizedShallow}
\end{equation}
where $\phi = g \h$ is the geopotential height with $g$ and $\h$ being
the gravitational constant and the perturbation of the free surface
height, $\Phi > 0$ is a constant mean flow geopotential height, $\vel := \LRp{\u,\v}$ is
the perturbed velocity, $\gamma \ge 0$ is the bottom friction, 
$\bs{\tau}:=\LRp{\tau_x,\tau_y}$ is the wind stress, and $\rho$ is the density of
the water. Here, $f = f_0 + \beta \LRp{y - y_m}$ is the
Coriolis parameter,  where $f_0$, $\beta$, and $y_m$ are given constants. We study the iHDG-II methods 
for this equation and compare it with the results 
in \cite{iHDG}.

For the simplicity of the exposition and the analysis, let us employ
the backward Euler HDG discretization for
\eqnref{linearizedShallow}. Since the unknowns of interest are those
at the $(m+1)$th time step, we can suppress the time index for the
clarity of the exposition.  Furthermore, since the system is linear it is sufficient to consider
homogeneous system with zero initial condition, zero boundary condition,
and zero forcing (wind stress).
Applying the iHDG-II algorithm \ref{al:DDMhyperbolic} to the
homogeneous system gives 
\begin{subequations}
\eqnlab{localSolverDD}
\begin{align}\nonumber
&\LRp{\frac{\phikp}{\Delta t},\varphi_1}_K - \LRp{\Phi\velkp,\Grad\varphi_1}_K + \LRa{\Phi \velkp \cdot \n + \sqrt{\Phi}\LRp{\phikp -\phih^{k,k+1}},\varphi_1}_\pK \\\eqnlab{localSolverPhiDD}
&= 0,\\ \nonumber
&\LRp{\frac{{\Phi\ukp}}{\Delta t},\varphi_2}_K - \LRp{\Phi\phikp,\pp{\varphi_2}{x}}_K + \LRa{\Phi\phih^{k,k+1}\n_1,\varphi_2}_\pK  \\\eqnlab{localSolverUDD}
&= \LRp{f\Phi\vkp - \gamma \Phi\ukp, \varphi_2}_K, \\ \nonumber
&\LRp{\frac{{\Phi\vkp}}{\Delta t},\varphi_3}_K - \LRp{\Phi\phikp,\pp{\varphi_3}{y}}_K + \LRa{\Phi\phih^{k,k+1}\n_2,\varphi_3}_\pK \\\label{localSolverVDD} 
&= \LRp{-f\Phi\ukp - \gamma \Phi\vkp , \varphi_3}_K,
\end{align}
\end{subequations}
where $\varphi_1, \varphi_2$ and $\varphi_3$ are the test functions, and
\begin{equation}
  \label{phihDD}
  \phih^{k,k+1}= \frac{1}{2}\LRc{\LRp{\phikp}^-+\LRp{\phik}^+} + \frac{\sqrt{\Phi}}{2}\LRc{\LRp{\velkp}^-\cdot\n^{-}+\LRp{\velk}^+\cdot\n^+}.
\end{equation}


\begin{lemma} The local solver \eqnref{localSolverDD} of the iHDG-II algorithm for the linearized shallow water equation is well-posed.
\end{lemma}
\begin{proof}
    Since $\LRc{\LRp{\phik}^+,\Phi\LRp{\velk}^+}$ is a ``forcing" to the local solver it is sufficient to set them to $\LRc{{0,{\bf 0}}}$ and show that the only solution possible is $\LRc{\LRp{\phik}^-,\Phi\LRp{\velk}^-}=\LRc{{0,{\bf 0}}}$.
Choosing the test functions $\varphi_1 =\phikp$, $\varphi_2 = \ukp$
and $\varphi_3 = \vkp$ in \eqnref{localSolverDD}, integrating the second term in \eqnref{localSolverPhiDD} by parts, and then summing equations in \eqnref{localSolverDD} altogether, we obtain

\begin{align}\nonumber
    &\frac{1}{\Delta t}\LRp{\phikp,\phikp}_K+\frac{\Phi}{\Delta t}\LRp{\velkp,\velkp}_K+\sqrt\Phi\LRa{\phikp,\phikp}_\pK
    +\gamma\Phi\LRp{\velkp,\velkp}_\K \\\label{localSolverTest}
    &-\sqrt\Phi\LRa{\phih^{k,k+1},\phikp}_\pK+\Phi\LRa{\phih^{k,k+1},{\bf\n}\cdot{\velkp}}_\pK=0.
\end{align}
Summing \eqref{localSolverTest} over all elements yields
\begin{align}\nonumber
    &\sum_\K\frac{1}{\Delta t}\LRp{\phikp,\phikp}_K+\frac{\Phi}{\Delta t}\LRp{\velkp,\velkp}_K+\gamma\Phi\LRp{\velkp,\velkp}_\K \\\label{SumlocalSolverTest}
    &+\sqrt\Phi\LRa{\phikp,\phikp}_\pK-\sqrt\Phi\LRa{\phih^{k,k+1},\phikp}_\pK+\Phi\LRa{\phih^{k,k+1},{\bf\n}\cdot{\velkp}}_\pK=0.
\end{align}
Substituting \eqref{phihDD} in the above equation and cancelling some terms we get,
\begin{align} \nonumber
    &\sum_\K\frac{1}{\Delta t}\nor{\LRp{\phikp}^-}_\K^2+\LRp{\gamma+\frac{1}{\Delta t}}\nor{\LRp{\velkp}^-}_{\Phi,\K}^2+\frac{\sqrt\Phi}{2}\nor{\LRp{\phikp}^-}_\pK^2 \\ \nonumber 
    &+\frac{\sqrt{\Phi}}{2}\nor{\LRp{\velkp\cdot\n}^-}_{\Phi,\pK}^2  
=\sum_\pK\frac{\sqrt\Phi}{2}\LRa{\LRc{\LRp{\phik}^++\sqrt{\Phi}\LRp{\velk\cdot\n}^+},\LRp{\phikp}^-}_\pK\\
\eqnlab{localsolver1}
&-\frac{\Phi}{2}\LRa{\LRc{\LRp{\phik}^++\sqrt{\Phi}\LRp{\velk\cdot\n}^+},\LRp{\velkp\cdot\n}^-}_\pK.
\end{align}

Since $\Phi>0$, all the terms on the left hand side are 
positive. When we set $\LRc{\LRp{\phik}^+,\Phi\LRp{\velk}^+}=\LRc{0,{\bf 0}}$, i.e. the data from  neighboring elements, the only solution possible is 
$\LRc{\LRp{\phikp}^-,\Phi\LRp{\velkp}^-}=\LRc{0,{\bf 0}}$ and hence the method is well-posed.
\end{proof}

Next, our goal is to show that $\LRp{\phikp,\Phi\velkp}$ converges to
zero. To that end, let us define
\begin{eqnarray}\eqnlab{ContractionConstant}
\mathcal{C}:=\frac{\mathcal{A}}{\mathcal{B}}, \quad \mathcal{A}
:=\frac{\max\LRc{1,\Phi}+\sqrt\Phi}{4\varepsilon}, \quad
\mathcal{G}
:=\frac{\varepsilon\LRp{\max\LRc{1,\Phi}+\sqrt\Phi}}{4}
\end{eqnarray}
and
\begin{align*}
    \mathcal{B}_1:=\left(\frac{ch}{\Delta t(\p+1)(\p+2)}+\frac{2\sqrt\Phi -(\Phi+\sqrt\Phi)\varepsilon}{4}\right)\\
    \mathcal{B}_2:=\left(\left(\gamma+\frac{1}{\Delta t}\right)\frac{ch}{(\p+1)(\p+2)}+\frac{2\sqrt\Phi-(1+\sqrt\Phi)\varepsilon}{4}\right),
\mathcal{B}
&:=\min\LRc{\mathcal{B}_1,\mathcal{B}_2}.
\end{align*}
where $0<c\le1$, $\varepsilon>0$ are constants.
We also need the following norms:
\begin{align*}
\nor{\LRp{\phik,\velk}}^2_{\Omega_h}&:=\nor{\phik}_{\Omega_h}^2 + \nor{\velk}_{\Phi, \Omega_h}^2, \\
\nor{\LRp{\phik,\velk\cdot\n}}^2_\Gh&:=\nor{\phik}_{\Gh}^2 + \nor{\velk\cdot\n}_{\Phi, \Gh}^2.\nonumber
\end{align*}

\begin{theorem}\theolab{ThmConvergenceShallowWater} Assume that the mesh size $h$, the time step $\Delta t$ and the solution order $\p$ are chosen such that $\mc{B} > 0$ and $\mathcal{C}<1$,
then the approximate solution at the $k$th iteration
$\LRp{\phik,\velk}$ converges to zero in the following sense
\begin{align}\nonumber
    \nor{\LRp{\phik,\velk\cdot\n}}^2_\Gh &\leq \mathcal{C}^k \nor{\LRp{\phi^0,\vel^0\cdot\n}}^2_\Gh,\\ \nonumber
    \nor{\LRp{\phik,\velk}}^2_{\Omega_h} &\leq \Delta t \LRp{\mc{A}+\mc{G}\mc{C}}\mc{C}^{k-1}\nor{\LRp{\phi^0,\vel^0\cdot\n}}^2_\Gh, \nonumber
\end{align}
where $\mc{C}$, $\mc{A}$ and $\mc{G}$ are defined in \eqnref{ContractionConstant}.
\end{theorem}
\begin{proof}
Using Cauchy-Schwarz and Young's inequalities for the terms on the right hand side of \eqnref{localsolver1} and simplifying

\begin{align} \nonumber
    &\sum_\K\frac{1}{\Delta t}\nor{\LRp{\phikp}^-}_\K^2+\LRp{\gamma+\frac{1}{\Delta t}}\nor{\LRp{\velkp}^-}_{\Phi,\K}^2+\frac{\sqrt\Phi}{2}\nor{\LRp{\phikp}^-}_\pK^2\\\nonumber 
    &+\frac{\sqrt{\Phi}}{2}\nor{\LRp{\velkp\cdot\n}^-}_{\Phi,\pK}^2\leq\sum_\pK\frac{\Phi+\sqrt\Phi}{4\varepsilon}\nor{\LRp{\phik}^+}_\pK^2\\\nonumber
    &+\frac{1+\sqrt\Phi}{4\varepsilon}\nor{\LRp{\velk\cdot\n}^+}_{\Phi,\pK}^2+\frac{\varepsilon(\Phi+\sqrt\Phi)}{4}\nor{\LRp{\phikp}^-}_\pK^2\\
    \eqnlab{localsolver2}
    &+\frac{\varepsilon(1+\sqrt\Phi)}{4}\nor{\LRp{\velkp\cdot\n}^-}_{\Phi,\pK}^2.
\end{align}

An application of inverse trace inequality \cite{chan2016gpu} for tensor product elements gives
\begin{subequations}\label{FEMTraceIn}
\begin{align}
    \LRp{\phikp,\phikp}_K&\geq\frac{2ch}{d(\p+1)(\p+2)}\LRa{\phikp,\phikp}_\pK,\\
    \LRp{\velkp,\velkp}_K&\geq\frac{2ch}{d(\p+1)(\p+2)}\LRa{\velkp,\velkp}_\pK,
\end{align}
\end{subequations}
where $d$ is the spatial dimension which in this case is $2$ and $0<c\le1$ is a constant.\footnote{Note that for simplices we can use the trace inequalities in \cite{MR1986022} and it will change only the constants in the proof.} Inequality \eqref{FEMTraceIn}, together with \eqnref{localsolver2}, implies
 \begin{eqnarray}\label{CompareTraces}\nonumber
     &&\sum_\pK\left[\left(\frac{ch}{\Delta t(\p+1)(\p+2)}+\frac{2\sqrt\Phi -(\Phi+\sqrt\Phi)\varepsilon}{4}\right)\nor{\LRp{\phikp}^-}_\pK^2\right.\\\nonumber
     &&\left.+\left(\left(\gamma+\frac{1}{\Delta t}\right)\frac{ch}{(\p+1)(\p+2)}+\frac{2\sqrt\Phi-(1+\sqrt\Phi)\varepsilon}{4}\right)\nor{\LRp{\velkp\cdot\n}^-}_{\Phi,\pK}^2\right]\\
     &\leq&\sum_\pK\left[\frac{\Phi+\sqrt\Phi}{4\varepsilon}\nor{\LRp{\phik}^+}_\pK^2+\frac{1+\sqrt\Phi}{4\varepsilon}\nor{\LRp{\velk\cdot\n}^+}_{\Phi,\pK}^2\right],
\end{eqnarray}
which then implies
 \begin{eqnarray}\label{Contraction}\nonumber
     \nor{\LRp{\phi^{k+1},\vel^{k+1}\cdot\n}}^2_\Gh \le \mc{C}\nor{\LRp{\phi^{k},\vel^{k}\cdot\n}}^2_\Gh,
\end{eqnarray}
where the constant $\mathcal{C}$ is computed as in \eqnref{ContractionConstant}. 
Therefore 
\begin{eqnarray}\label{TraceExponentialDecay}
\nor{\LRp{\phikp,\velkp\cdot\n}}^2_\Gh \le \mc{C}^{k+1}\nor{\LRp{\phi^0,\vel^0\cdot\n}}^2_\Gh.
\end{eqnarray}
On the other hand, inequalities \eqnref{localsolver2} and \eqref{TraceExponentialDecay} imply
\begin{align}\nonumber
    \nor{\LRp{\phikp,\velkp}}^2_{\Omega_h} &\le \Delta t \LRp{\mc{A}+\mc{G}\mc{C}}\mc{C}^k\nor{\LRp{\phi^0,\vel^0\cdot\n}}^2_\Gh
\end{align}
    and this ends the proof.
\end{proof}

We now derive an explicit relation between the number of iterations
$k$, the meshsize $h$, the solution order $p$, the time step $\Delta t$ and the mean flow geopotential height $\Phi$. First, we need to find an $\varepsilon$ which makes $\mc{C}<1$. From \eqnref{ContractionConstant} we obtain the following inequality for $\varepsilon$  
\begin{equation}
\label{contraction_equation}
    \frac{\frac{\max\LRc{1,\Phi}+\sqrt\Phi}{4\varepsilon}}{\left(\frac{ch}{\Delta t(\p+1)(\p+2)}+\frac{2\sqrt\Phi -(\max\LRc{1,\Phi}+\sqrt\Phi)\varepsilon}{4}\right)}<1.
\end{equation}
A sufficient condition for the denominator to be positive and existence of a real $\varepsilon>0$ to the above inequality \eqref{contraction_equation} is
\begin{equation}
\label{coarse_constraint}
    \frac{\frac{ch}{\Delta t (\p+1)(\p+2)}}{\max\LRc{1,\Phi}+\sqrt{\Phi}}>\frac{1}{2}.
\end{equation}
This allows us to find an $\varepsilon>0$ that satisfies the
inequality \eqref{contraction_equation} for all $\Phi$. In particular, we can pick
\begin{equation}
  \eqnlab{epsilon}
    \varepsilon=\frac{\frac{2ch}{\Delta t (\p+1)(\p+2)}+\sqrt{\Phi}}{\max\LRc{1,\Phi}+\sqrt{\Phi}}.
\end{equation}
Using this value of $\varepsilon$ in equation
\eqnref{ContractionConstant} we get
\[
    \mc{C}=\LRp{\frac{\frac{\max\LRc{1,\Phi}+\sqrt{\Phi}}{\sqrt{\Phi}}}{1+\frac{2ch}{\sqrt{\Phi}\Delta t (\p+1)(\p+2)}}}^2
\]
and since the numerator is always greater than $1$, the necessary and sufficient condition for the convergence of the algorithm is given by
\[
    \frac{1}{\LRp{1+\frac{2ch}{\sqrt{\Phi}\Delta t (\p+1)(\p+2)}}^{2k}}\stackrel{k\to \infty}{\longrightarrow} 0.
\]
Using binomial theorem and neglecting higher order terms we get
\begin{equation}
\label{khp_estimate}
k=\mc{O}\LRp{\frac{\Delta t (p+1)(p+2)\sqrt{\Phi}}{4ch}}.
\end{equation}
Note that if we choose $\Delta t$ similar to explicit method, i.e.  $\Delta t =  \mc{O}\LRp{ \frac{h}{\p^2\sqrt{\Phi}}}$ \cite{TAYLOR199792}, $k=\mc{O}(1)$ independent of $h$ and $p$.
With this result in hand we are now in a better position to 
understand the stability of iHDG-I and iHDG-II algorithms for the linearized shallow water system.
For unconditional stability of the iterative algorithms
under consideration, we need $\mc{B}>0$ in \eqnref{ContractionConstant} independent of
$h,p$ and $\Delta t$. There are two terms in $\mc{B}$:
$\mc{B}_1$ coming from $\phi$ and $\mc{B}_2$ from $\vel$ or 
$\vel \cdot \n$.
We can write $\mc{B}$ in Theorem 3.6
of \cite{iHDG} for iHDG-I also as\footnote{This can be obtained
    by using Young's inequality with $\varepsilon$ in 
the proof of Theorem 3.6 in \cite{iHDG}.}
\begin{align}\label{B1_iHDG1}
    \mathcal{B}_1:=\left(\frac{ch}{\Delta t(\p+1)(\p+2)}+\frac{2\sqrt\Phi -(\Phi+\sqrt{\Phi})\varepsilon}{2}\right)\\\label{B2_iHDG1}
    \mathcal{B}_2:=\left(\left(\gamma+\frac{1}{\Delta t}\right)\frac{ch}{(\p+1)(\p+2)}-\frac{(1+\sqrt\Phi)\varepsilon}{2}\right),
\mathcal{B}
&:=\min\LRc{\mathcal{B}_1,\mathcal{B}_2}.
\end{align}
Note that for both iHDG-I and iHDG-II algorithms we have the stability
in $\phi$ independent of $h,p$ and $\Delta t$, since we can choose
$\varepsilon$ sufficiently small independent of $h,p$ and $\Delta t$
and make $\mc{B}_1>0$ in \eqref{B1_iHDG1} and
\eqnref{ContractionConstant}. However, from \eqref{B2_iHDG1} we have
to choose $\varepsilon$ as a function of $(h,p,\Delta t)$ in order to
have $\mc{B}_2>0$, and hence iHDG-I lacks the mesh independent
stability in the term associated with $\vel$. This explains the
instability observed in \cite{iHDG} for fine meshes and/or large time steps. Since $\mc{B}_2$ in \eqnref{ContractionConstant} can be made positive with a
sufficiently small $\varepsilon$, independent of $h,p$ and $\Delta t$,
iHDG-II is always stable: a significant advantage over  iHDG-I.

\section{iHDG-II for convection-diffusion PDEs}
\seclab{iHDG-convection-diffusion}
\subsection{First order form}
In this section we apply the iHDG-II algorithm \ref{al:DDMhyperbolic} to the following first order form 
of the convection-diffusion equation:
\begin{subequations}
\eqnlab{Convection-diffusion-eqn}
\begin{align}
\kappa^{-1}\sigb^e + \Grad \u^e &= 0 \quad \text{ in } \Omega, \\ 
\Div\sigb^e +\betab \cdot \Grad \u^e +\nu \u^e &= f \quad \text{ in }
\Omega.
\end{align}
\end{subequations}
We assume that \eqnref{Convection-diffusion-eqn} is well-posed, i.e.,
\begin{equation}
	\label{CondCoercivity}
	\nu-\frac{\Div \betab}{2}\geq \lambda > 0.
\end{equation}
Though this is not a restriction, we take constant diffusion coefficient for the simplicity of the exposition. 
An upwind HDG numerical flux \cite{Bui-Thanh15} is given by
\[
\Fh\cdot \n = 
\LRs{
\begin{array}{c}
\uh \n_1
\\ \uh \n_2
\\ \uh \n_3
\\ \sigb \cdot \n + \betab \cdot \n \u + \tau\LRp{\u -\uh}
\end{array}
},
\]
where $\tau=\frac{1}{2}\LRp{\alpha-\betab\cdot\n}$ and $\alpha=\sqrt{|\betab\cdot\n|^2+4}$. 
Similar to the previous sections, it is sufficient to
consider the homogeneous problem. Applying the iHDG-II algorithm \ref{al:DDMhyperbolic} we have the following iterative scheme

\begin{subequations}\label{ErrorDDMCD}
\begin{align}
\label{ErrorDDMCD1}
\kappa^{-1}\LRp{\sigbkp,\taub}_K-\LRp{\ukp,\nabla\cdot\taub}_K+\LRa{\uh^{k,k+1},\taub\cdot\n}_\pK&=0,\\
\label{ErrorDDMCD2}
-\LRp{\sigbkp,\nabla\v}_K-\LRp{\ukp,\nabla\cdot\LRp{\betab\v}-\nu{\v}}_K&+\nonumber\\
\LRa{\betab\cdot\n\ukp+\sigbkp\cdot\n+\tau(\ukp-\uh^{k,k+1}),\v}_\pK&=0,
\end{align}
\end{subequations}
where
\begin{multline}
    \label{uhatCDR}
    \uh^{k,k+1}=\frac{\LRc{\LRp{\sigbkp\cdot\n}^{-}+\LRp{\sigbk\cdot\n}^{+}}+\LRc{\betab\cdot\n^{-}\LRp{\ukp}^-+\betab\cdot\n^{+}\LRp{\uk}^+}}{\alpha}\\
    +\frac{\LRc{\tau^{-}\LRp{\ukp}^-+\tau^{+}\LRp{\uk}^+}}{\alpha}.
\end{multline}

\begin{lemma} The local solver \eqref{ErrorDDMCD} of the iHDG-II algorithm for the convection-diffusion equation is well-posed.
\end{lemma}

\begin{proof}
    The proof is similar to the one for the shallow water equation and hence is given in the Appendix A \secref{appendixa}. 
\end{proof}
       
Now, we are in a position to prove the convergence of the algorithm.
For $\varepsilon$, $h>0$ and $0<c\le1$ given, define
\begin{align}\label{DefineC1}
\mathcal{C}_1&:=\frac{(\|\betab\cdot{\bf n}\|_{L^\infty(\pK)}+\bar\tau)(\bar\tau+1)}{2\varepsilon\alpha_*}, \, \mathcal{C}_2:=\frac{(\bar\tau+1)}{2\varepsilon\alpha_*},\\\label{DefineC2}
    \mathcal{C}_3&:=\frac{\varepsilon\bar\tau(1+\bar\tau+\|\betab\cdot{\bf n}\|_{L^\infty(\pK)})}{2\alpha_*}, \, \mathcal{C}_4:=\frac{\varepsilon(1+\bar\tau+\|\betab\cdot{\bf n}\|_{L^\infty(\pK)})}{2\alpha_*},\\\label{DefineD}
    \mc{D}&:= \frac{\mc{A}}{\mc{B}}, \, \mc{A}={\max\{\mathcal{C}_1,\mathcal{C}_2\}}, \, \mc{E}:=\frac{\max\{\mathcal{C}_3,\mathcal{C}_4\}}{\min\{\kappa^{-1},\lambda\}}, \, \mc{F}:=\frac{\mc{A}}{\min\{\kappa^{-1},\lambda\}},\\\label{DefineB_12}
    \mc{B}_1&:=\frac{2ch\kappa^{-1}}{d(p+1)(p+2)}+\frac{1}{2\bar\alpha}-\mathcal{C}_4,
    \mc{B}_2:=\frac{2ch\lambda}{d(p+1)(p+2)}+\frac{1}{\bar\alpha}-\mathcal{C}_3,\\\label{DefineB}
    \mc{B}&:={\min\{\mc{B}_1,\mc{B}_2\}}.
\end{align}
As in the previous section we need the following norms 
\begin{align}
\nor{\LRp{\sigb^k,\uk}}^2_{\Omega_h}&:=\nor{\sigb^k}_{\Omega_h}^2 + \nor{\uk}_{\Omega_h}^2, \quad \nor{\LRp{\sigb^k\cdot\n,\uk}}^2_\Gh&:=\nor{\sigb^k\cdot\n}_{\Gh}^2 + \nor{\uk}_{\Gh}^2.\nonumber
\end{align}

\begin{theorem} 
\theolab{ThmUpwindDDM}
Suppose that the mesh size $h$ and the solution order $p$ are chosen such that  $\mc{B}>0$ and $\mc{D}<1$, 
the algorithm  \eqref{ErrorDDMCD1}-\eqref{uhatCDR} converges in the following sense
\begin{align}\nonumber
    \nor{\LRp{\sigb^{k}\cdot\n,\uk}}^2_\Gh &\le \mc{D}^{k}\nor{\LRp{\sigb^{0}\cdot\n,\u^{0}}}^2_\Gh,\\\nonumber
    \nor{\LRp{\sigb^{k},\uk}}^2_{\Omega_h} &\le (\mc{E}\mc{D}+\mc{F})\mc{D}^{k-1}\nor{\LRp{\sigb^{0}\cdot\n,\u^{0}}}^2_\Gh, 
\end{align}
where $\mc{D}, \mc{E}$ and $\mc{F}$ are defined in \eqref{DefineD}. 
\end{theorem}

\begin{proof}
The proof is similar to the one for the shallow water equation and hence is given in the Appendix B \secref{appendixb}. 
\end{proof}

Similar to the discussion in section \ref{iiHDG_shallow}, one can show that
\begin{equation}
\eqnlab{khp_estimate_CDR}
k=\mc{O}\LRp{\frac{d(p+1)(p+2)}{8\bar\alpha ch \min\LRc{\kappa^{-1},\lambda}}}.
\end{equation}

For time-dependent convection-diffusion equation, we discretize the spatial differential operators using HDG. For the temporal
derivative, we use implicit time stepping methods, again with either backward Euler or Crank-Nicolson method
for simplicity. The analysis in this case is almost identical to
the one for steady state equation except that we now have an
additional $L^2$-term $\LRp{\ukp,\v}_{\K}/\Delta t$ in the local
equation \eqref{ErrorDDMCD2}. This improves the convergence of the iHDG-II method. Indeed, the convergence analysis is the same except we
now have $\lambda + 1/\Delta t$ in place of
$\lambda$. In particular we have the following estimation
\[
    k=\mc{O}\LRp{\frac{d(p+1)(p+2)}{8\bar\alpha ch \min\LRc{\kappa^{-1},\LRp{\lambda+1/\Delta t}}}}.
\]

\begin{remark}
  Similar to the shallow water system if we choose $\Delta t = \mc{O}\LRp{\frac{h}{\p^2}}$ then the number of 
iterations is independent of $h$ and $p$. 
This is more efficient than the iterative 
hybridizable IPDG method for the  
parabolic equation
in \cite{gander2016analysis}, which requires
$\Delta t = \mc{O}(\frac{h^2}{\p^4})$ in order to achieve constant iterations. 
The reason is perhaps due the fact that hybridizable IPDG is posed directly on the
second order form whereas HDG uses the first order form.
While iHDG-I has mesh independent stability for only $\u$ (see \cite[Theorem 4.1]{iHDG}), iHDG-II does for both $u$ and $\sigb$; an important improvement.
\end{remark}

\section{Numerical results}
\seclab{numerical}
In this section various numerical results verifying
the theoretical results are provided for the transport
equation, the linearized shallow water equation, and the
convection-diffusion equation in both two- and three-dimensions.

\subsection{Transport equation}
We consider the same 2D and 3D test 
cases in \cite[sections 5.1.1 and 5.1.2]{iHDG}. The domain  is an unit square/cube
with structured quadrilateral/hexahedral elements.
Throughout the numerical section, we use the following
stopping criteria 
\begin{equation}
\eqnlab{tolerancecriteria}
|\norm{u^k-u^e}_{\Ltwo}-\norm{u^{k-1}-u^e}_{\Ltwo}|<10^{-10},
\end{equation}
if the exact solution is available, and
\begin{equation}
\eqnlab{tolerancecriteriarelerror}
\norm{u^k-u^{k-1}}_{\Ltwo}<10^{-10},
\end{equation}
if the exact solution is not available.

From Theorem \theoref{DDMConvergence}, the theoretical number of
iterations is approximately $d \times (N_{el})^{1/d}$ (where $d$ is the 
dimension). It can be seen from the fourth and fifth columns of Table
\ref{tab:2d_discont_3d_smooth} that the numerical results agree well
with the theoretical prediction. We can also see that the number of
iterations is independent of solution order, which is consistent
with the theoretical result Theorem \theoref{DDMConvergence}. Figure \figref{shocksoln} shows the
solution converging from the inflow boundary to the outflow boundary
in a layer-by-layer manner. Again, the process is automatic, i.e., no
prior element ordering or information about the advection velocity is required.

Now, we study the parallel performance of the iHDG algorithm. For this
purpose we have implemented iHDG algorithm on top of {\it mangll}
\cite{WilcoxStadlerBursteddeEtAl10,BursteddeGhattasGurnisEtAl10,BursteddeGhattasGurnisEtAl08}
which is a higher order continuous/discontinuous finite element
library that supports large scale parallel simulations using MPI. The
simulations are conducted on {\bf Stampede} at the Texas Advanced
Computing Center (TACC).

\begin{figure}[h!t!b!]
  \subfigure[$\uk$ at $k = 16$]{
    \includegraphics[trim=1.0cm 0.5cm 1.5cm 1cm,clip=true,width=0.3\columnwidth]{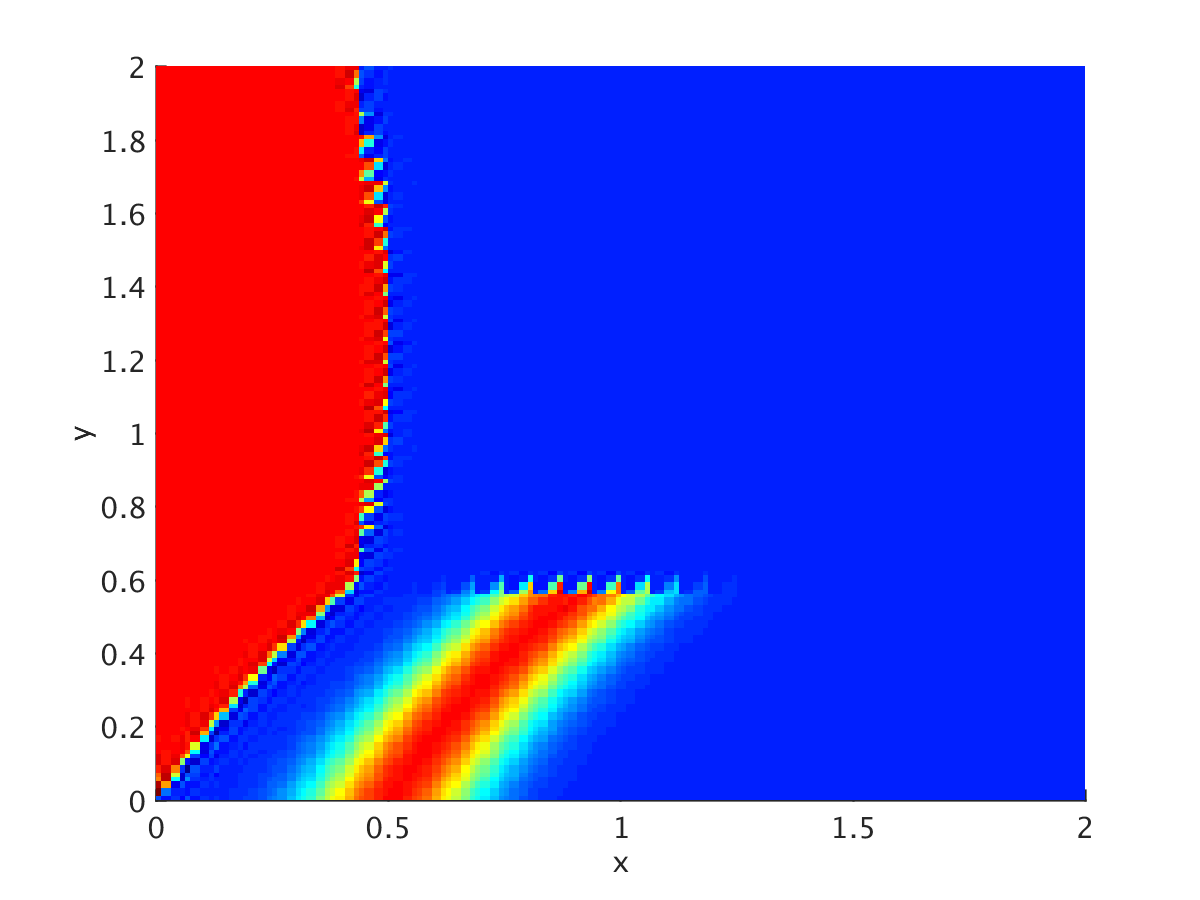}
  }
  \subfigure[$\uk$ at $k = 48$]{
    \includegraphics[trim=1.0cm 0.5cm 1.5cm 1cm,clip=true,width=0.3\columnwidth]{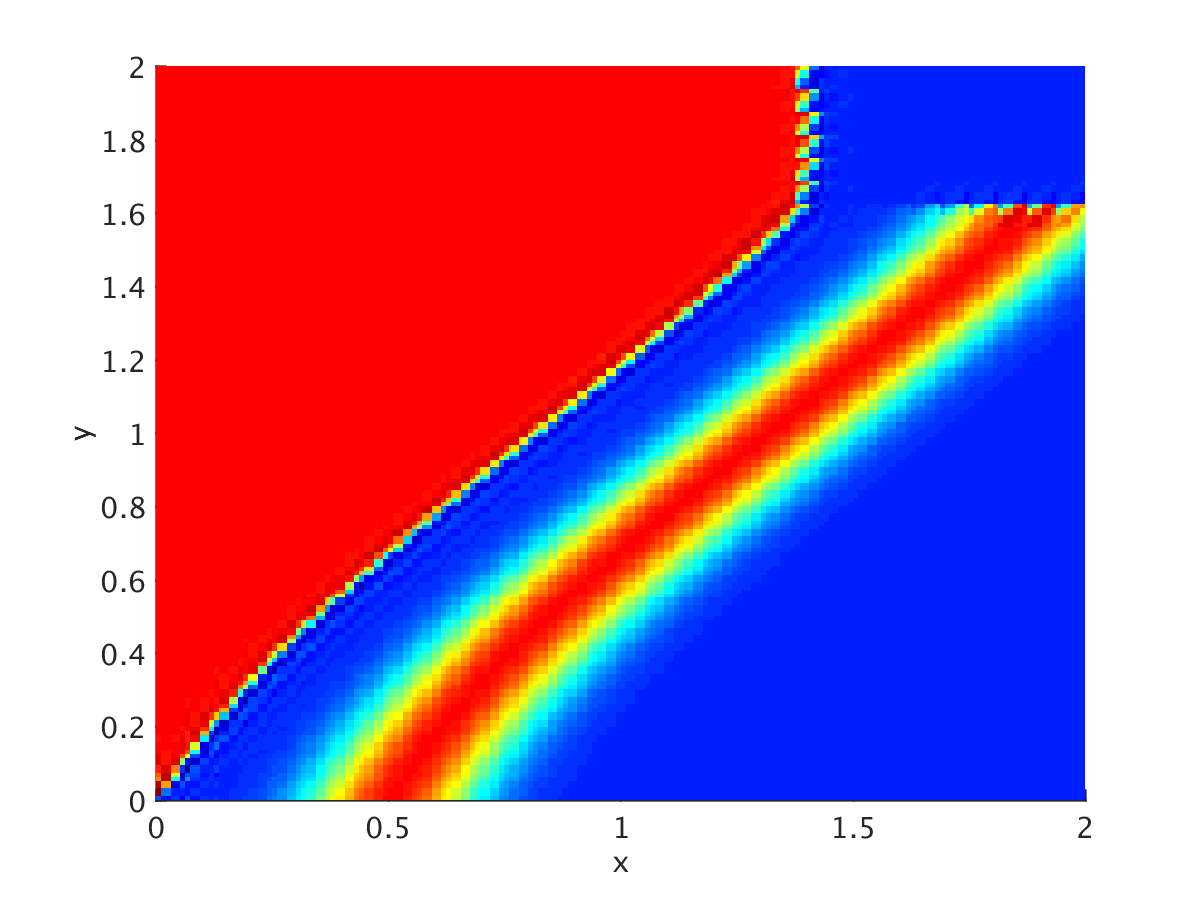}
  }
  \subfigure[$\uk$ at $k = 64$]{
    \includegraphics[trim=1.0cm 0.5cm 1.5cm 1cm,clip=true,width=0.3\columnwidth]{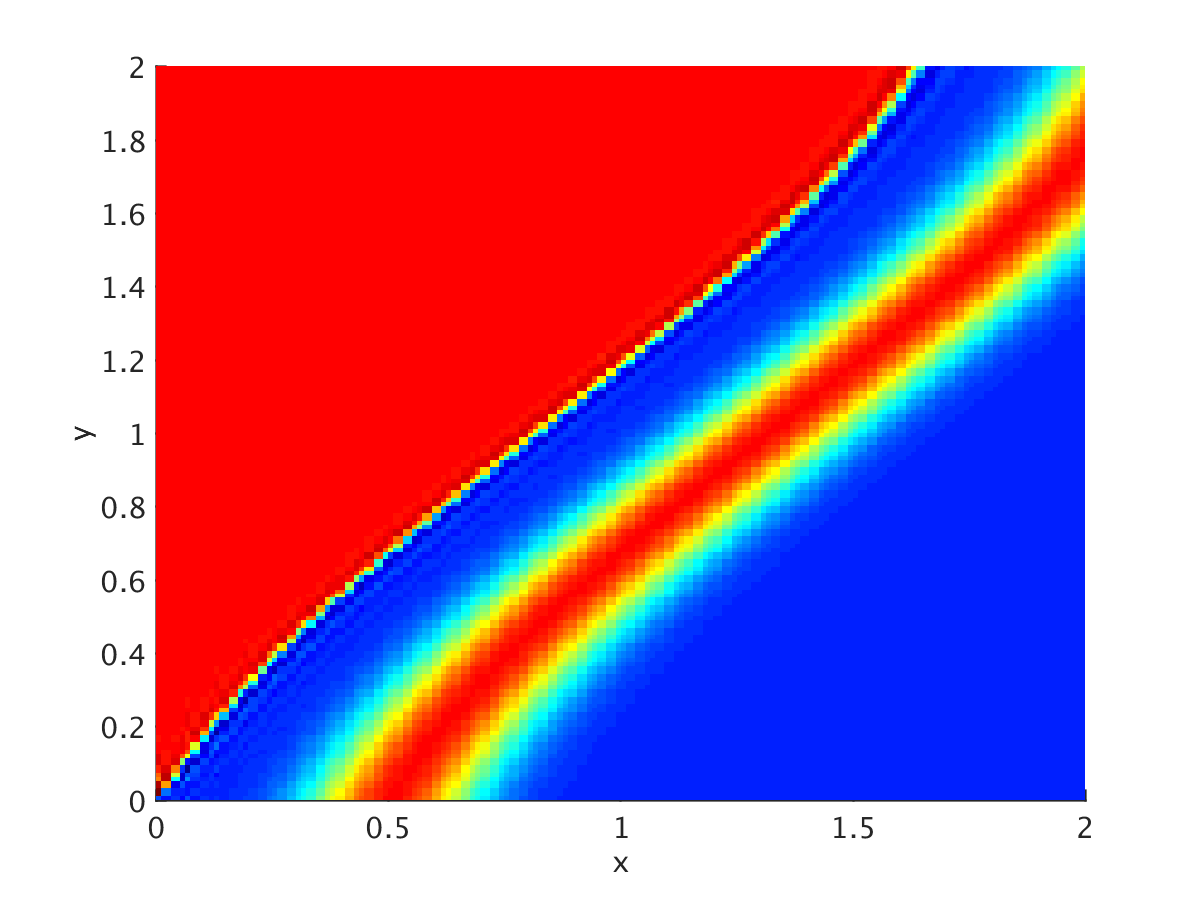}
  }
  \caption{Evolution of the iterative solution for the 2D transport equation using the iHDG-II algorithm.}
  \figlab{shocksoln}
\end{figure}

\begin{table}[h!b!t!]
\begin{center} 
\caption{The number of iterations taken by the
  iHDG-II algorithm for the transport equation in 2D and 3D settings.}
\label{tab:2d_discont_3d_smooth}
\begin{tabular}{ | r || r || r || c | c | }
\hline
{$N_{el}$(2D)} & {$N_{el}$(3D)} & {$p$} & {2D solution} & {3D solution} \\
\hline
4x4 & 2x2x2 & 1 & 9 & 6 \\
\hline
8x8 & 4x4x4 & 1 & 17 & 12 \\
\hline
16x16 & 8x8x8 & 1 & 33 & 23 \\
\hline
32x32 & 16x16x16 & 1 & 65 & 47 \\
\hline
\hline
4x4 & 2x2x2 & 2 & 9 & 6 \\
\hline
8x8 & 4x4x4 & 2 & 17 & 12 \\
\hline
16x16 & 8x8x8 & 2 & 33 & 23 \\
\hline
32x32 & 16x16x16 & 2 & 65 & 47 \\
\hline
\hline
4x4 & 2x2x2 & 3 & 9 & 7 \\
\hline
8x8 & 4x4x4 & 3 & 17 & 12 \\
\hline
16x16 & 8x8x8 & 3 & 33 & 23 \\
\hline
32x32 & 16x16x16 & 3 & 65 & 47 \\
\hline
\hline
4x4 & 2x2x2 & 4 & 9 & 6 \\
\hline
8x8 & 4x4x4 & 4 & 17 & 12 \\
\hline
16x16 & 8x8x8 & 4 & 33 & 24 \\
\hline
32x32 & 16x16x16 & 4 & 64 & 48 \\
\hline
\end{tabular}
\end{center}
\end{table}

Table \ref{tab:strong scaling} shows strong scaling 
results for the 3D transport problem. 
The parallel efficiency is over $90\%$ for all the cases
except for the case where we use 16,384 cores and 16 elements per core whose efficiency is $59\%$. This is due to the fact that, with 16 elements
per core, the communication cost dominates the computation.
Table \ref{tab:weak scaling} shows the weak scaling with
1024 and 128 elements/core. Since the number of iterations increases linearly with the number of elements,
we can see a similar increase in time when we increase
the number of elements, and hence cores.

Let  us now  consider the time dependent 3D transport equation
with the following exact solution
\[
	u^e=e^{-5((x-0.35t)^2+(y-0.35t)^2+(z-0.35t)^2)},
\]
where the velocity field is chosen to be $\betab=(0.2,0.2,0.2)$. In
Figure \ref{CFL_vs_k} are the numbers of iHDG iterations taken per
time step to converge versus the $CFL$ number. As can be seen, for
$CFL$ in the range $\LRs{1,5}$ the number of iterations grows
mildly. As a result, we get a much better weak scaling results in
Table \ref{tab:weak scaling time} in comparison to the steady state
case in Table \ref{tab:weak scaling}.

\begin{table}[h!b!t!]
\begin{center} 
\caption{Strong scaling results on TACC's Stampede system for the 3D transport problem.}
\label{tab:strong scaling}
\begin{tabular}{ | r | c | c | c | }
\hline
\multicolumn{4}{|c|}{$N_{el}=262,144$, $p=4$, dof=32.768 million, Iterations=190} \\
\hline
{\#cores} & {time [s]} & {$N_{el}$/core} & {efficiency [\%]} \\
\hline
64 & 1758.62 & 4096 & 100.0\\
\hline
128 & 883.88 & 2048 & 99.5 \\
\hline
256 & 439.94 & 1024 & 99.9 \\
\hline
512 & 228.69 & 512 & 96.1 \\
\hline
1024 & 113.87 & 256 & 96.5 \\
\hline
2048 & 56.36 & 128 & 97.5 \\
\hline
4096 & 29.26 & 64 & 91.8 \\
\hline
16384 & 11.38 & 16 & 59 \\
\hline
\multicolumn{4}{|c|}{$N_{el}=2,097,152$, $p=4$, dof=262.144 million, Iterations=382} \\
\hline
{\#cores} & {time [s]} & {$N_{el}$/core} & {efficiency [\%]} \\
\hline
512 & 3634.89 & 4096 & 100.0 \\
\hline
1024 & 1788.78 & 2048 & 101.5 \\
\hline
2048 & 932.495 & 1024 & 97.3 \\
\hline
4096 & 447.337 & 512 & 101.5 \\
\hline
8192 & 232.019 & 256 & 97.9 \\
\hline
16384 & 117.985 & 128 & 92.9 \\
\hline
\end{tabular}
\end{center}
\end{table}

\begin{table}[h!b!t!]
\begin{center} 
\caption{Weak scaling results on TACC's Stampede system for the 3D transport problem.}
\label{tab:weak scaling}
\begin{tabular}{ | r | c | c | c | }
\hline
\multicolumn{4}{|c|}{1024 $N_{el}$/core, $p=4$} \\
\hline
{\#cores} & {time [s]} & {time ratio} & {Iterations ratio} \\
\hline
4 & 103.93 & 1 & 1 \\
\hline
32 & 217.23 & 2.1 & 2 \\
\hline
256 & 439.94 & 4.2 & 4 \\
\hline
2048 & 932.49 & 8.9 & 8 \\
\hline
\multicolumn{4}{|c|}{128 $N_{el}$/core, $p=4$} \\
\hline
{\#cores} & {time [s]} & {time ratio} & {Iterations ratio} \\
\hline
4 & 6.52 & 1 & 1 \\
\hline
32 & 13.68 & 2.1 & 2 \\
\hline
256 & 27.71 & 4.2 & 4 \\
\hline
2048 & 56.37 & 8.6 & 8 \\
\hline
\end{tabular}
\end{center}
\end{table}

\begin{figure}[h!b!t!]
\centering
\includegraphics[trim=2cm 6.5cm 2cm 7.5cm,clip=true,width=0.7\textwidth]{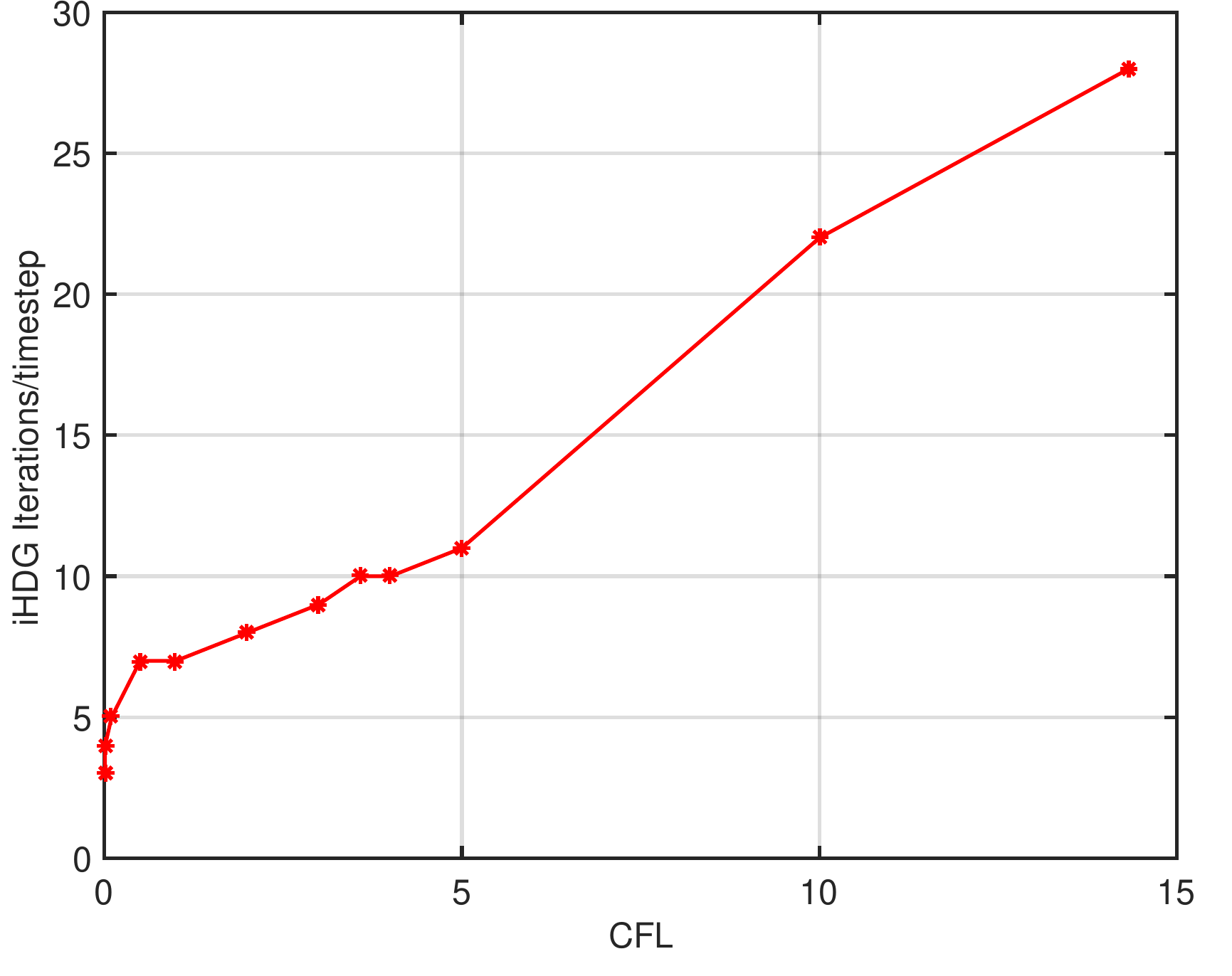}
\caption{CFL versus Iterations for the 3D time dependent transport}
\label{CFL_vs_k}
\end{figure}

\begin{table}[h!b!t!]
\begin{center} 
\caption{Weak scaling results on TACC's Stampede system for the 3D time dependent transport problem.}
\label{tab:weak scaling time}
\begin{tabular}{ | r | c | c | c | c | }
\hline
\multicolumn{5}{|c|}{128 $N_{el}$/core, $\p=4$, $\Delta t$=0.01, $|\betab|_{max}$=0.35} \\
\hline
{\#cores} & {time/timestep [s]} & {time ratio} & {Iterations ratio} & {CFL} \\
\hline
4 & 1.69 & 1 & 1 & 0.45 \\
\hline
32 & 1.91 & 1.1 & 1.1 & 0.9 \\
\hline
256 & 2.09 & 1.2 & 1.1 & 1.8 \\
\hline
2048 & 2.72 & 1.6 & 1.4 & 3.6 \\
\hline
16384 & 4.68 & 2.8 & 2.1 & 7.2 \\
\hline
\end{tabular}
\end{center}
\end{table}

\subsection{Linearized shallow water equations}
The goal of this section is to verify the theoretical
findings in section \ref{iiHDG_shallow}. To that extent, let us consider equation \eqnref{linearizedShallow} with a linear
standing wave, for which, $\Phi=1$, $f=0$, $\gamma=0$ (zero bottom friction),
$\taub=0$ (zero wind stress). The domain is $[0, 1]\times[0, 1]$ and
the wall boundary condition is applied on the domain boundary.
The following exact
solution \cite{GiraldoWarburton08} is taken
\begin{subequations}
\eqnlab{shallowexact}
\begin{align}
    \phi^{e}&=\cos(\pi x)\cos(\pi y)\cos(\sqrt{2} \pi t), \\
    u^{e}&=\frac{1}{\sqrt{2}}\sin(\pi x)\cos(\pi y)\sin(\sqrt{2} \pi t), \\
    v^{e}&=\frac{1}{\sqrt{2}}\cos(\pi x)\sin(\pi y)\sin(\sqrt{2} \pi t). 
\end{align}
\end{subequations}
We use Crank-Nicolson method for the temporal discretization and the iHDG-II approach
for the spatial discretization. 
In Table
\ref{tab:Comparison_iHDG_iiHDG_shallow_water} we compare the number of
iterations taken by iHDG-I and iHDG-II methods for two different time
steps $\Delta t = 0.1$ and $\Delta t = 0.01$. 
Here, ``$*$'' indicates divergence. As can be seen from the
third and fourth columns, the iHDG-I method diverges for finer meshes and/or
larger time steps. This is consistent with the findings in section
\ref{iiHDG_shallow} where the divergence is expected because of
the lack of mesh independent stability in the velocity. On the
contrary, iHDG-II converges for all cases. 


\begin{table}[h!b!t!]
\begin{center}
\caption{Comparison of iHDG-I and iHDG-II
for the shallow water equations.}
\label{tab:Comparison_iHDG_iiHDG_shallow_water}
\begin{tabular}{ | r || r || c | c | c | c | }
\hline
\multirow{2}{*}{$N_{el}$} & \multirow{2}{*}{$p$} & \multicolumn{2}{|c|}{iHDG-I} & \multicolumn{2}{|c|}{iHDG-II} \\
& & $\Delta t=10^{-1}$ & $\Delta t=10^{-2}$ & $\Delta t=10^{-1}$ & $\Delta t=10^{-2}$ \\
\hline
16 & 1 & 19 & 6 & 14 & 6 \\
\hline
64 & 1 & * & 6 & 18 & 9 \\
\hline
256 & 1 &* & 7 & 32 & 10 \\
\hline
1024 & 1 & * & 9 & 59 & 8 \\
\hline
\hline
16 & 2 & * & 9 & 15 & 9 \\
\hline
64 & 2 &  * & 11 & 19 & 9 \\
\hline
256 & 2 & * & 13 & 32 & 11 \\
\hline
1024 & 2 & * & 15 & 59 & 12 \\
\hline
\hline
16 & 3 & * & 7 & 16 & 8 \\
\hline
64 & 3 & * & 9 & 20 & 8 \\
\hline
256 & 3 & * & 12 & 31 & 10 \\
\hline
1024 & 3 & * & * & 59 & 12 \\
\hline
\hline
16 & 4 &  * & 10 & 17 & 9 \\
\hline
64 & 4 & * & 12 & 32 & 10 \\
\hline
256 & 4 & * & * & 60 & 9 \\
\hline
1024 & 4 & * & * & 112 & 13 \\
\hline
\end{tabular}
\end{center}
\end{table}


\begin{figure}[h!b!t!]
\begin{minipage}{0.5\linewidth}
\centering
\includegraphics[trim=1.5cm 5cm 2cm 6cm,clip=true,width=0.8\textwidth]{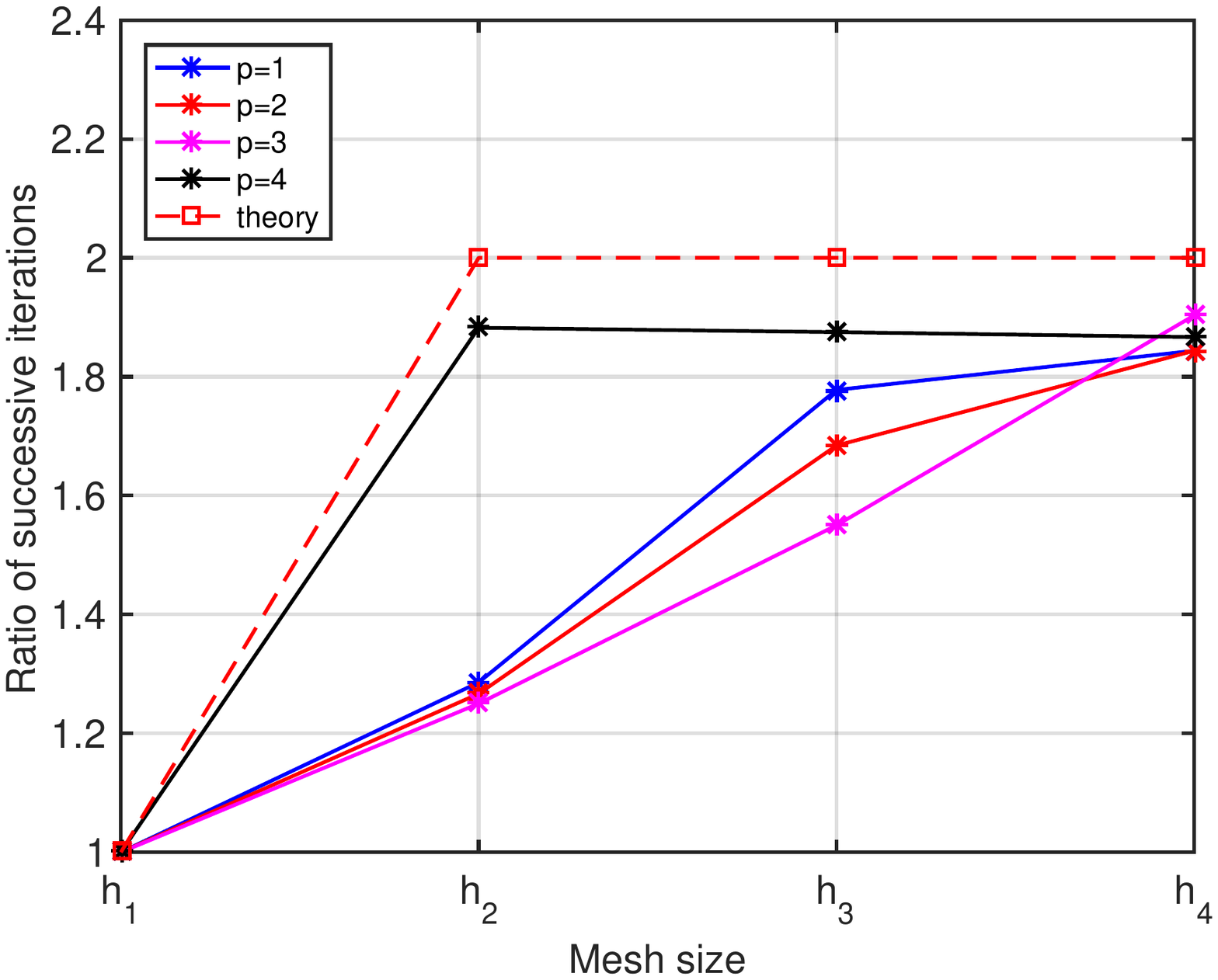}
\end{minipage}
\begin{minipage}{0.5\linewidth}
\begin{center}
\begin{tabular}{ | r || c | c | c | c | c | }
\hline
\multirow{2}{*}{$\p$} & \multicolumn{4}{|c|}{iHDG-II} & \multirow{2}{*}{Asymptotics} \\
& $h_1$ & $h_2$ & $h_3$ & $h_4$ & \\
\hline
$2$ & 1.07 & 1.06 & 1 & 1 & 2\\
\hline
$3$ & 1.14 & 1.11 & 0.97 & 1 & 3.33 \\
\hline
$4$ & 1.21 & 1.78 & 1.87 & 1.9 & 5 \\
\hline
\end{tabular}
\end{center}
\end{minipage}
\caption{Growth of iterations with mesh size $h$ (left) and solution order $p$ (right) for the iHDG-II method for the shallow water equation.}\label{tab:p_var}
\end{figure}

In Table \ref{tab:Comparison_iHDG_iiHDG_shallow_water}, we use a
series of structured quadrilateral meshes with uniform refinements
such that the ratio of successive mesh sizes is $1/2$. The asymptotic
result \eqref{khp_estimate}, which is valid for $\frac{h}{\Delta t
  (p+1)(p+2)\sqrt{\Phi}}\ll 1$, predicts that the ratio of the number
of iterations required by successive refined meshes is $2$, and the
results in Figure \ref{tab:p_var} confirm this prediction. The last two columns of Table
\ref{tab:Comparison_iHDG_iiHDG_shallow_water} also confirms the
asymptotic result \eqref{khp_estimate} that the number of iHDG-II iterations scales linearly with the time stepsize.

Next, we study the iHDG iteration growth as the
solution order $p$ increases. 
The asymptotic result \eqref{khp_estimate} predicts that $k = {O}(p^2)$. In Table 
\ref{tab:p_var}, rows 2--4 show the ratio of the number of iterations taken for 
solution orders $p = \LRc{2,3,4}$ over the one for $p = 1$ for four different 
meshsizes as in Table \ref{tab:Comparison_iHDG_iiHDG_shallow_water}. As can be seen, the
theoretical estimation is conservative.


\subsection{Convection-diffusion equation}
\seclab{CDR_numerical}
In this section the following exact solution for equation \eqnref{Convection-diffusion-eqn} is considered 
$$u^e=\frac{1}{\pi}\sin(\pi x)\cos(\pi y)\sin(\pi z).$$ The forcing is
chosen such that it corresponds to the exact solution. The velocity field is chosen as $\betab=(1+z,1+x,1+y)$ 
and we take  $\nu=1$. The domain is $[0,1]\times[0,1]\times[0,1]$. A structured hexahedral mesh is used for the simulations. The stopping criteria based on the exact solution is used as in the 
previous sections.

In Table \ref{tab:Comparison_iHDG_iiHDG_CDR} we report the number of iterations taken by iHDG-I and iHDG-II methods for different values of the diffusion coefficent $\kappa$. Similar to the shallow water
equations, due to the lack of stability in $\sigb$, iHDG-I diverges when $\kappa$
is large for fine meshes and/or high solution orders. The iHDG-II method, on the other hand,
converges for all the meshes and solution orders, and the number of iterations are smaller  than that of the iHDG-I method. Next, we verify the growth of iHDG-II iterations predicted by the asymptotic result \eqnref{khp_estimate_CDR}.

Since $\min\LRc{\kappa^{-1},\lambda} = \lambda$
  for all the numerical results considered here, due to
  \eqnref{khp_estimate_CDR} we expect the number of iHDG-II iterations to be independent of $\kappa$ and this can be verified in Table
  \ref{tab:Comparison_iHDG_iiHDG_CDR}. In Figure \ref{h_var_CDR} the growth of
  iterations with respect to mesh size $h$ for different $\kappa$ are compared. In the asymptotic
  limit, for all the cases, the ratio of successive iterations reaches
  a value of around $1.7$ which is close to the theoretical prediction
  $2$. 
  Hence the theoretical analysis predicts well the growth of
  iterations with respect to the mesh size $h$. On the other hand, columns $6-8$
  in Table \ref{tab:Comparison_iHDG_iiHDG_CDR} show that the iterations are
  almost independent of solution orders. This is not predicted by the
  theoretical results which indicates that  the number of  iterations scales like  $\mc{O}(p^2)$. The reason is due to the convection
  dominated nature of the problem, for which we have shown that the number of iterations is independent of  the
  solution order.

Finally, we consider the elliptic regime with $\kappa=1$ and
$\betab=0$. For this case we use the following stopping criteria based on
the direct solver solution $\u_{direct}$
\[
    \nor{\u^{k}-\u_{direct}}_{L^2}<10^{-10}.
\]
As shown in Figure \ref{tab:h_p_var_elliptic}, our theoretical
analysis predicts well the relation between the number of iterations and the mesh size and the solution order.  

 \begin{table}[h!b!t!]
 \begin{center}
     \caption{Comparison of iHDG-I and iHDG-II methods for different $\kappa$.} 
 \label{tab:Comparison_iHDG_iiHDG_CDR}
 \begin{tabular}{ | r || r || c | c | c | c | c | c | }
 \hline
 \multirow{2}{*}{$h$} & \multirow{2}{*}{$p$} & \multicolumn{3}{|c|}{iHDG-I} & \multicolumn{3}{|c|}{iHDG-II} \\
 & & $\kappa=10^{-2}$ & $\kappa=10^{-3}$ & $\kappa=10^{-6}$ & $\kappa=10^{-2}$ & $\kappa=10^{-3}$ & $\kappa=10^{-6}$ \\
 \hline
 0.5 & 1 & 24 & 23 & 23 & 17 & 17 & 17 \\
 \hline
 0.25 & 1 & 30 & 34 & 35 & 25 & 25 & 26 \\
 \hline
 0.125 & 1 & 50 & 55 & 56 & 35 & 37 & 38 \\
 \hline
 0.0625 & 1 & 90 & 94 & 97 & 62 & 64 & 65 \\
 \hline
 \hline
 0.5 & 2 & 26 & 24 & 25 & 17 & 19 & 19 \\
 \hline
 0.25 & 2 & 41 & 42 & 42 & 27 & 27 & 27 \\
 \hline
 0.125 & 2 & 66 & 67 & 67 & 42 & 43 & 43 \\
 \hline
 0.0625 & 2 & * & 109 & 110 & 67 & 70 & 71 \\
 \hline
 \hline
 0.5 & 3 & 27 & 31 & 31 & 19 & 19 & 19 \\
 \hline
 0.25 & 3 & 33 & 33 & 38 & 24 & 26 & 27 \\
 \hline
 0.125 & 3 & * & 58 & 60 & 38 & 39 & 41 \\
 \hline
 0.0625 & 3 & * & 102 & 106 & 69 & 69 & 71 \\
 \hline
 \hline
 0.5 & 4 & 26 & 27 & 27 & 17 & 19 & 19 \\
 \hline
 0.25 & 4 & 50 & 41 & 43 & 26 & 27 & 27 \\
 \hline
 0.125 & 4 & * & 71 & 72 & 42 & 45 & 46 \\
 \hline
 0.0625 & 4 & * & 123 & 125 & 73 & 78 & 79 \\
 \hline
 \end{tabular}
 \end{center}
 \end{table}

\begin{figure}[h!t]
\vspace{-5mm}
\subfigure[$\kappa = 10^{-2}$]{
    \includegraphics[trim=1.5cm 6cm 2cm 6cm,clip=true,width=0.3\textwidth]{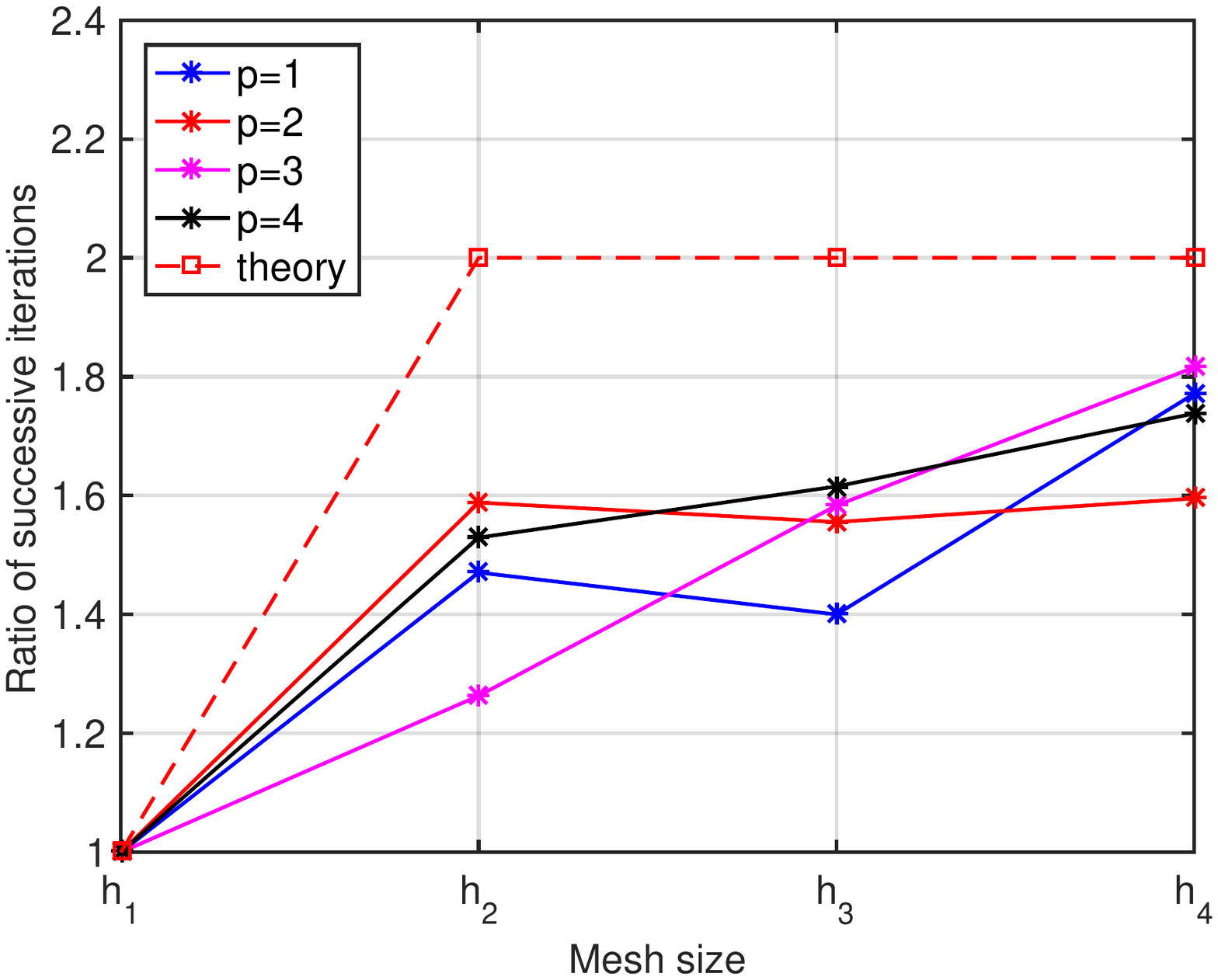}
    \label{k1em2}
}
\subfigure[$\kappa = 10^{-3}$]{
    \includegraphics[trim=1.5cm 6cm 2cm 6cm,clip=true,width=0.3\textwidth]{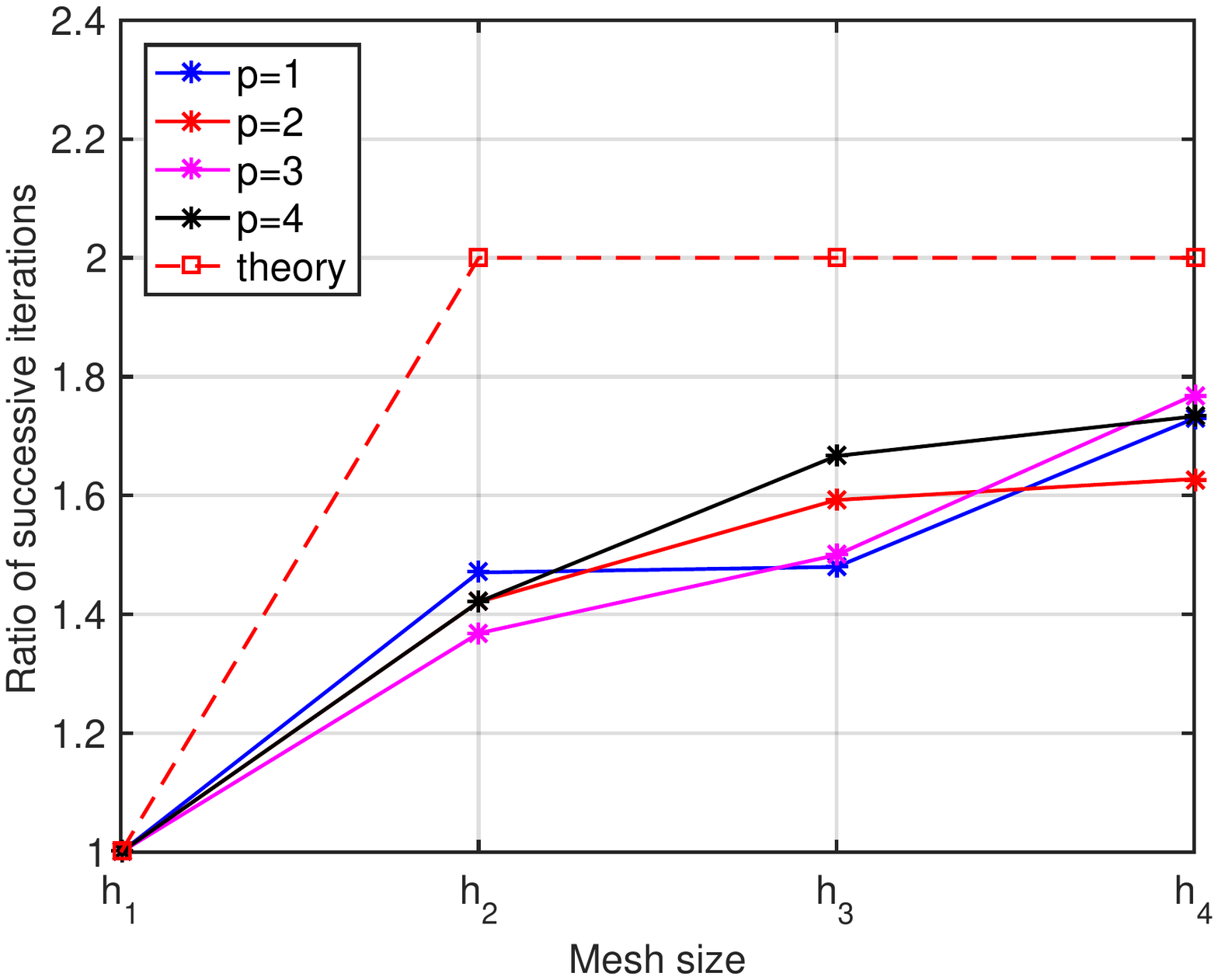}
\label{k1em3}
}
\subfigure[$\kappa = 10^{-6}$]{
    \includegraphics[trim=1.5cm 6cm 2cm 6cm,clip=true,width=0.3\textwidth]{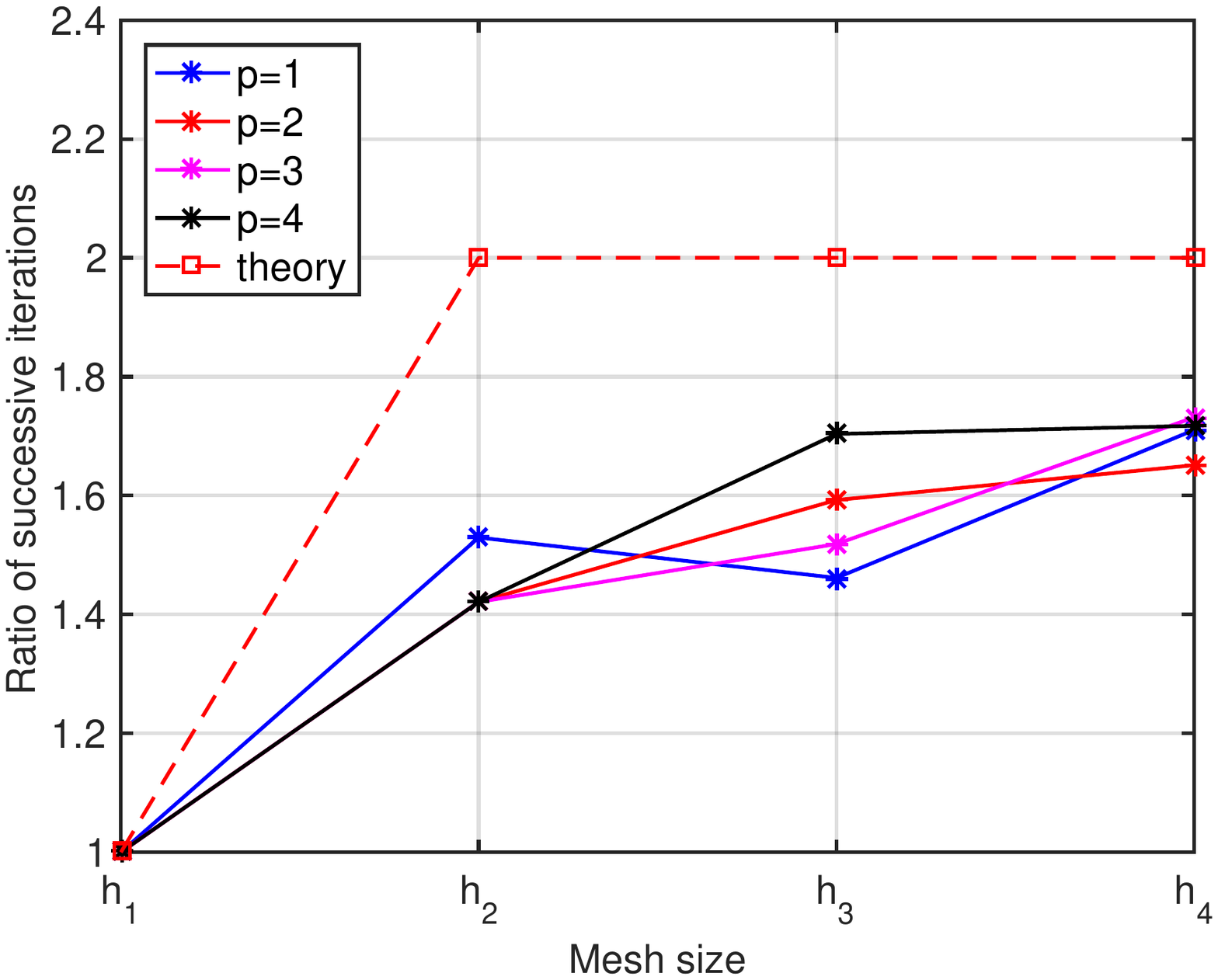}
\label{k1em6}
}
\caption{Ratio of successive iterations as we refine the
mesh for the iHDG-II method for different $\kappa$.} 
\label{h_var_CDR}
\end{figure}

\begin{figure}[h!b!t!]
\begin{minipage}{0.5\linewidth}
\centering
    \includegraphics[trim=1.5cm 6cm 2cm 6cm,clip=true,width=0.8\textwidth]{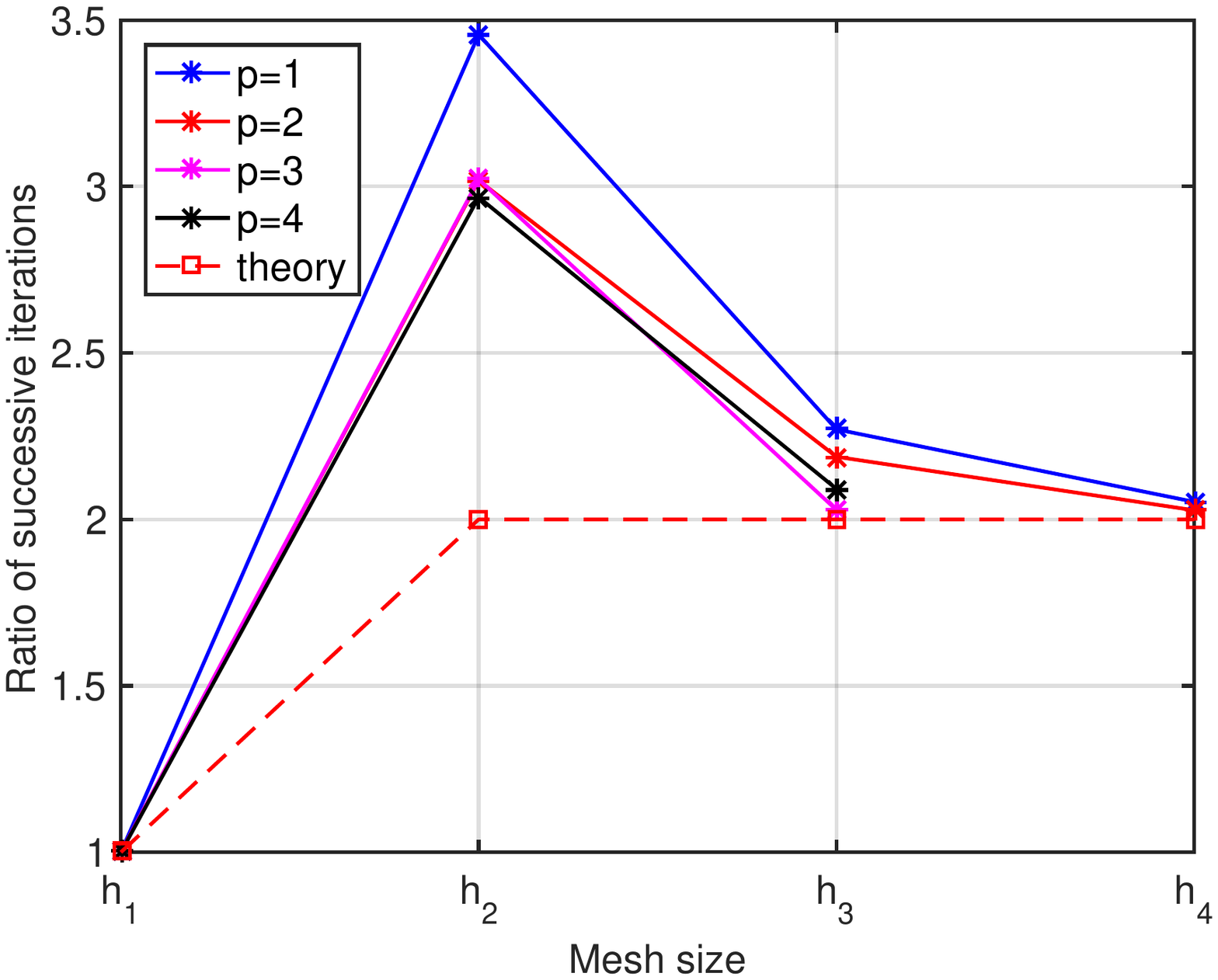}
\end{minipage}
\begin{minipage}{0.5\linewidth}
\begin{center}
\begin{tabular}{ | r || c | c | c | c | }
\hline
\multirow{2}{*}{$\p$} & \multicolumn{3}{|c|}{iHDG-II} & \multirow{2}{*}{Asymptotics} \\
& $h_1$ & $h_2$ & $h_3$ & \\
\hline
$2$ & 2.41 & 2.11 & 2.03 & 2\\
\hline
$3$ & 4.07 & 3.55 & 3.17 & 3.33 \\
\hline
$4$ & 6.09 & 5.23 & 4.81 & 5 \\
\hline
\end{tabular}
\end{center}
\end{minipage}
\caption{Growth of iterations with mesh size $h$ (left) and solution order $p$ (right) for the iHDG-II method for the elliptic equation.}\label{tab:h_p_var_elliptic}
\end{figure}

\section{Conclusions}
\seclab{conclusions}

We have presented an iterative HDG approach which improves upon our
previous work \cite{iHDG} in several aspects.  In particular, it
converges in a finite number of iterations for the scalar transport
equation and is unconditionally convergent for both the linearized shallow
water system and the convection-diffusion equation. Moreover, compared to
our previous work \cite{iHDG}, we provide several additional findings:
1) we make a connection between iHDG and the parareal method, which
reveals interesting similarities and differences between the two
methods; 2) we show that iHDG can be considered as a {\em locally
  implicit} method, and hence being somewhat in between fully explicit
and fully implicit approaches; 3) for both the linearized shallow water system
and the convection-diffusion equation, using an asymptotic approximation,
we uncover a relationship between the number of iterations and time
stepsize, solution order, meshsize and the equation parameters. This
allows us to choose the time stepsize such that the number of
iterations is approximately independent of the solution order and the
meshsize; 4) we show that iHDG-II has improved stability and
convergence rates over iHDG-I; and 5) we provide both strong and weak
scalings of our iHDG approach up to $16,384$ cores. Ongoing work is
to develop a preconditioner to reduce the number of iterations as the
mesh is refined. Equally important is to exploit the iHDG approach
with small number of iterations as a preconditioner.


\section*{Acknowledgements}
 We are indebted to Professor Hari Sundar for sharing his high-order
finite element library {\bf homg}, on which we have implemented the iHDG algorithms and produced numerical results. We thank Texas Advanced Computing Center (TACC) for processing our requests so quickly regarding running our simulations on Stampede.
The first author would like to thank Stephen Shannon for helping in proving some results and various fruitful discussions on this topic. We thank Professor Martin Gander 
for pointing us to the reference \cite{gander2008analysis}. The first author would like to dedicate this work in memory of his grandfather Gopalakrishnan Thiruvenkadam. 

\section*{Appendix A. Proof of well-posedness of local solver of the iHDG-II method for the convection-diffusion equation}
\seclab{appendixa}
\begin{proof}
Choosing $\sigbkp$ and $\ukp$ as test functions in
\eqref{ErrorDDMCD1}-\eqref{ErrorDDMCD2}, integrating the second term
in \eqref{ErrorDDMCD1} by parts, using \eqnref{Identity} for second term in \eqref{ErrorDDMCD2}, 
and then summing up the resulting two equations we obtain
\begin{multline}
    \kappa^{-1}\LRp{\sigb^{k+1},\sigb^{k+1}}_K+\LRp{\LRc{\nu-\frac{\Div\betab}{2}}\ukp,\ukp}_K\\
    +\LRa{\LRp{\frac{\betab\cdot{\bf n}}{2}+\tau}\ukp,\ukp}_\pK
    +\LRa{(\sigb^{k+1}\cdot{\bf n}-\tau\ukp),\uh^{k,k+1}}_\pK=0.
\label{ErrorDDMCD4}
\end{multline}
Substituting \eqref{uhatCDR} in the above equation and simplifying some terms we get,
\begin{align}\nonumber
    &\sum_K\kappa^{-1}\nor{\LRp{\sigbkp}^-}_K^2+\LRp{\LRc{\nu-\frac{\Div\betab}{2}}\LRp{\ukp}^-,\LRp{\ukp}^-}_K\\\nonumber
    &+\LRa{\LRc{\frac{|\betab\cdot\n|^2+2}{2\alpha}}\LRp{\ukp}^-,\LRp{\ukp}^-}_\pK
    +\LRa{\frac{1}{\alpha}\LRp{\sigbkp\cdot\n}^-,\LRp{\sigbkp\cdot\n}^-}_\pK\\\nonumber
    &+\LRa{\frac{\betab\cdot\n^-}{\alpha}\LRp{\ukp}^-,\LRp{\sigbkp\cdot\n}^-}_\pK
=\sum_\pK-\LRa{\frac{1}{\alpha}\LRp{\sigbkp\cdot\n}^-,\LRp{\sigbk\cdot\n}^+}_\pK\\\nonumber
&-\LRa{\LRc{\frac{\betab\cdot\n^{+}+\tau^+}{\alpha}}\LRp{\sigbkp\cdot\n}^-,\LRp{\uk}^+}_\pK+\LRa{\frac{\tau^-}{\alpha}\LRp{\ukp}^-,\LRp{\sigbk\cdot\n}^+}_\pK\\\label{ErrorDDMCD5}
&+\LRa{\LRc{\frac{\tau^-(\betab\cdot\n^++\tau^+)}{\alpha}}\LRp{\ukp}^-,\LRp{\uk}^+}_\pK.
\end{align}
Using the identity
\begin{multline}
    \LRa{\frac{\betab\cdot\n}{\alpha}\ukp,\sigb^{k+1}\cdot\n}_\pK=\nor{\frac{1}{\sqrt{2\alpha}}\LRp{\betab\cdot\n\ukp+\sigb^{k+1}\cdot\n}}_\pK^2\\
    -\LRa{\frac{\betab\cdot\n^2}{2\alpha}\ukp,\ukp}_\pK
-\LRa{\frac{1}{2\alpha}\sigb^{k+1}\cdot\n,\sigb^{k+1}\cdot\n}_\pK,
\label{completesquare}
\end{multline}
and the coercivity condition \eqref{CondCoercivity} we
can write \eqref{ErrorDDMCD5} as
\begin{align}\nonumber
    &\sum_K\kappa^{-1}\nor{\LRp{\sigbkp}^-}_K^2+\lambda\nor{\LRp{\ukp}^-}_K^2
    +\nor{\LRp{\ukp}^-}_{1/\alpha,\pK}^2\\\nonumber
    &+\nor{\LRp{\sigbkp\cdot\n}^-}_{1/2\alpha,\pK}^2
    +\nor{\frac{1}{\sqrt{2\alpha}}\LRc{\betab\cdot\n^-\LRp{\ukp}^-+\LRp{\sigbkp\cdot\n}^-}}_\pK^2
\leq \\\nonumber
&\sum_\pK-\LRa{\frac{1}{\alpha}\LRp{\sigbkp\cdot\n}^-,\LRp{\sigbk\cdot\n}^+}_\pK
-\LRa{\LRc{\frac{\betab\cdot\n^{+}+\tau^+}{\alpha}}\LRp{\sigbkp\cdot\n}^-,\LRp{\uk}^+}_\pK\\\label{wellposedcdr}
&+\LRa{\frac{\tau^-}{\alpha}\LRp{\ukp}^-,\LRp{\sigbk\cdot\n}^+}_\pK+\LRa{\LRc{\frac{\tau^-(\betab\cdot\n^++\tau^+)}{\alpha}}\LRp{\ukp}^-,\LRp{\uk}^+}_\pK.
\end{align}
Since all the terms on the left hand side are positive, when we take the ``forcing" to the local solver $\LRc{\LRp{\uk}^+,\LRp{\sigbk}^+}=\LRc{0,{\bf 0}}$, the only solution possible is 
$\LRc{\LRp{\ukp}^-,\LRp{\sigbkp}^-}=\LRc{0,{\bf 0}}$ and hence the method is well-posed. 
\end{proof}

\section*{Appendix B. Proof of convergence of the iHDG-II method for the convection-diffusion equation}
\seclab{appendixb}
\begin{proof}
    In equation \eqref{wellposedcdr} omitting the last term on the left hand side and using Cauchy-Schwarz and Young's inequalities for the terms on the right hand side we get
\begin{align}\nonumber
    &\sum_K\kappa^{-1}\nor{\LRp{\sigbkp}^-}_K^2+\lambda\nor{\LRp{\ukp}^-}_K^2
    +\nor{\LRp{\ukp}^-}_{1/\alpha,\pK}^2\\\nonumber
    &+\nor{\LRp{\sigbkp\cdot\n}^-}_{1/2\alpha,\pK}^2
    \leq\sum_\pK\frac{1}{2\varepsilon}\LRa{\LRc{\frac{\tau^-+1}{\alpha}}\LRp{\sigbk\cdot\n}^+,\LRp{\sigbk\cdot\n}^+}_\pK\\\nonumber
    &+\frac{1}{2\varepsilon}\LRa{\LRc{\frac{(1+\tau^-)(\betab\cdot\n^{+}+\tau^+)}{\alpha}}\LRp{\uk}^+,\LRp{\uk}^+}_\pK\\\nonumber
    &+\frac{\varepsilon}{2}\LRa{\LRc{\frac{1+\tau^++\betab\cdot\n^+}{\alpha}}\LRp{\sigbkp\cdot\n}^-,\LRp{\sigbkp\cdot\n}^-}_\pK\\\label{ErrorDDMCD7}
    &+\frac{\varepsilon}{2}\LRa{\LRc{\frac{\tau^-(1+\tau^++\betab\cdot\n^+)}{\alpha}}\LRp{\ukp}^-,\LRp{\ukp}^-}_\pK.
\end{align}
We can write the above inequality as
\begin{align}\nonumber
    &\kappa^{-1}\nor{\LRp{\sigbkp}^-}_K^2+\lambda\nor{\LRp{\ukp}^-}_K^2
    +\frac{1}{\bar\alpha}\nor{\LRp{\ukp}^-}_{\pK}^2
    +\frac{1}{2\bar\alpha}\nor{\LRp{\sigbkp\cdot\n}^-}_{\pK}^2\\\nonumber
    &\leq \frac{\bar\tau+1}{2\varepsilon\alpha_*}\nor{\LRp{\sigbk\cdot\n}^+}_\pK^2+\frac{(1+\bar\tau)(\|\betab\cdot\n\|_{L^{\infty}(\pK)}+\bar\tau)}{2\varepsilon\alpha_*}\nor{\LRp{\uk}^+}_\pK^2\\\nonumber   
    &+\frac{\varepsilon(1+\bar\tau+\|\betab\cdot\n\|_{L^{\infty}(\pK)})}{2\alpha_*}\nor{\LRp{\sigbkp\cdot\n}^-}_\pK^2\\\label{ErrorDDMCD8}
    &+\frac{\varepsilon\bar\tau(1+\bar\tau+\|\betab\cdot\n\|_{L^{\infty}(\pK)})}{2\alpha_*}\nor{\LRp{\ukp}^-}_\pK^2,
\end{align}
where $\bar\tau := \|\tau\|_{L^\infty(\pOmega_h)}$, $\bar\alpha := \|\alpha\|_{L^\infty(\pOmega_h)}$, and $\alpha_* := \inf\limits_{\pK \in \pOmega_h}\alpha$. 

By the inverse trace inequality \eqref{FEMTraceIn} we infer from \eqref{ErrorDDMCD8} that
\begin{multline}\nonumber
    \sum_\pK\mc{B}_1\nor{\LRp{\sigbkp\cdot\n}^-}_{\pK}^2+\mc{B}_2\nor{\LRp{\ukp}^-}_{\pK}^2\\
    \leq\sum_\pK\LRs{\mathcal{C}_1\nor{\LRp{\uk}^+}_{\pK}^2+\mathcal{C}_2\nor{\LRp{\sigbk\cdot\n}^+}_{\pK}^2},
\end{multline}
which implies
 \begin{eqnarray}\nonumber
    \nor{\LRp{\sigb^{k+1}\cdot\n,\ukp}}^2_\Gh \le \mc{D}\nor{\LRp{\sigb^{k}\cdot\n,\uk}}^2_\Gh,
\end{eqnarray}
where the constant $\mc{D}$ is computed as in \eqref{DefineD}.
Therefore 
\begin{eqnarray}\label{TraceExponentialDecay_CD}
    \nor{\LRp{\sigb^{k+1}\cdot\n,\ukp}}^2_\Gh \le \mc{D}^{k+1}\nor{\LRp{\sigb^{0}\cdot\n,\u^{0}}}^2_\Gh.
\end{eqnarray}
Inequalities \eqref{ErrorDDMCD8} and \eqref{TraceExponentialDecay_CD} imply
\begin{eqnarray}\nonumber
    \nor{\LRp{\sigb^{k+1},\ukp}}^2_{\Omega_h} \le (\mc{E}\mc{D}+\mc{F})\mc{D}^{k}\nor{\LRp{\sigb^{0}\cdot\n,\u^{0}}}^2_\Gh,
\end{eqnarray}
and this concludes the proof.

\end{proof}
\section*{References}

\bibliographystyle{elsarticle-num}

\bibliography{references,ceo}

\end{document}